\documentclass[11pt, leqno]{amsart}

\usepackage{graphicx}
\usepackage{amssymb,amsmath,amsthm,amscd}
\usepackage{mathrsfs}
\usepackage{enumerate}
\usepackage{lscape}
\usepackage[all]{xy}
\usepackage{verbatim}

\newtheorem{theorem}{Theorem}[section]
\newtheorem{proposition}[theorem]{Proposition}

\newtheorem{lemma}[theorem]{Lemma}

\theoremstyle{definition}
\newtheorem{definition}[theorem]{Definition}

\newtheorem{remark}[theorem]{Remark}
\newtheorem{assumption}[theorem]{Assumption}

\numberwithin{equation}{section}

\newcommand{\soplus}{\mathop{\mbox{\normalsize$\bigoplus$}}\limits}

\newlength{\mylength}
\newcommand{\nc}{\newcommand}
\newcommand{\ko}{\mathbf{k}}
\nc{\Hom}{\mathrm{Hom}}
\nc{\cl}{\mathrm{cl}}
\nc{\tens}{\mathop\otimes}
\nc{\wl}{\mathsf{P}}
\nc{\rl}{\mathsf{Q}}
\nc{\cm}{\mathsf{A}}
\nc{\s}{\mathsf{s}}
\nc{\Rnorm}{R^{\rm{norm}}}
\nc{\Runiv}{R^{\rm{univ}}}
\nc{\aff}{\mathrm{aff}}
\def\Rep{{\rm Rep}}

\def\g{\mathfrak g}
\def\h{\mathfrak h}
\nc{\Dkl}{D_{k,l}}
\nc{\mqs}{(-q^2)}
\def\Z{\mathbb Z}
\def\C{\mathbb C}
\def\Q{\mathbb Q}
\def\va{\varpi}
\def\seteq{\mathbin{:=}}

\nc\uqpg{U_q'(\g)}

\title[The Denominators of normalized $R$-matrices] {The Denominators of normalized $R$-matrices of
types \\ $A_{2n-1}^{(2)}$, $A_{2n}^{(2)}$, $B_{n}^{(1)}$ and
$D_{n+1}^{(2)}$}

\author[Se-jin Oh]{Se-jin Oh}

\address{Department of Mathematical Sciences, Seoul National University,
Gwanak-ro 599, 
Gwanak-gu, Seoul 151-747, Korea}
\email{sejin092@gmail.com}

\keywords{Quantum affine algebra, Normalized $R$-matrix}

\subjclass[2010] {81R50, 16G, 16T25, 17B37}

\date{\today}

\begin{document}

\begin{abstract} Denominators of normalized $R$-matrices provide important information on finite dimensional integrable representations
over quantum affine algebras, and over quiver Hecke algebras by the generalized quantum affine Schur-Weyl duality functors.
We compute the denominators of all normalized $R$-matrices between fundamental representations of types $A_{2n-1}^{(2)}$, $A_{2n}^{(2)}$, $B_{n}^{(1)}$ and
$D_{n+1}^{(2)}$. Thus we can conclude that the normalized $R$-matrices of types $A_{2n-1}^{(2)}$, $A_{2n}^{(2)}$, $B_{n}^{(1)}$ have only simple poles,
and of type $D_{n+1}^{(2)}$ have double poles under certain conditions.
\end{abstract}

\maketitle

\section*{Introduction}

Let $\g$ be an affine Kac-Moody algebra and $U_q'(\g)$ be the quantum affine algebra corresponding to $\g$.
The finite dimensional integrable representations over $U_q'(\g)$ have been investigated by many authors during the past twenty years from different
perspectives (see \cite{AK,CH,CP94,FR,GV,Kas02,Nak}). Among these aspects, we focus on the theory of {\it $R$-matrices} which has deep relationship with
$q$-analysis, operator algebras, conformal field theories, statistical mechanical models, etc.

The purpose of this paper is to compute the denominators of {\it normalized $R$-matrices} between the fundamental representations
$V(\varpi_{i_k})$'s over $U_q'(\g)$. Knowing the denominators is quite crucial to study the finite dimensional integrable
representations by the following theorem:
\bigskip

\noindent
\textbf{Theorem} \cite{AK,Kas02}
Let $M$ be a finite dimensional irreducible integrable $U'_q(\g)$-module $M$. Then, there
exists a finite sequence $$\big( (i_1,a_1),\ldots, (i_l,a_l)\big) \text{ in }
( \{ 1,2,\cdots,n \} \times \ko^\times)^l$$ such that
\begin{itemize}
\item $d_{i_k,i_{k'}}(a_{k'}/a_k) \not=0$ for $1\le k<k'\le l$ and
\item $M$ is isomorphic to the head of $\tens_{i=1}^{l}V(\varpi_{i_k})_{a_k}$.
\end{itemize}
Moreover, such a sequence $\left((i_1,a_1),\ldots, (i_l,a_l)\right)$ is
unique up to permutation. Here $\ko =\overline{\C(q)} \subset \cup_{m >0} \C((q^{1/m}))$ and $d_{i_k,i_{k'}}(z) \in \ko[z]$
denotes the denominator of the normalized $R$-matrix
$$\Rnorm_{i_k,i_{k'}}(z)\seteq\Rnorm_{V(\va_{i_k}),V(\va_{i_{k'}})}(z) \colon V(\va_{i_k}) \otimes  V(\va_{i_{k'}})_z \to \ko(z) \otimes_{\ko[z^{\pm 1}]} \big(V(\va_{i_{k'}})_z \otimes V(\va_{i_k}) \big)$$
satisfying
$$d_{i_k,i_{k'}}(z)\Rnorm_{i_k,i_{k'}}(z)\big(  V(\va_{i_k}) \otimes  V(\va_{i_{k'}})_z  \big) \subset V(\va_{i_{k'}})_z \otimes V(\va_{i_k}).$$
\medskip

Thus the study of denominators is one of the first step to study the category $\mathscr{C}_\g$ consisting of finite dimensional integrable representations over
$U_q'(\g)$.

On the other hand, Kang, Kashiwara and Kim \cite{KKK13a,KKK13b}
recently constructed {\it the quantum affine Schur-Weyl duality
functor} $\mathcal{F}$ by observing zeros of denominators of
normalized $R$-matrices. The way of constructing $\mathcal{F}$ can
be described as follows: Let $\{ V_s \}_{s \in \mathcal{S}}$ be a
family of fundamental representations over $U'_q(\g)$. For an index
set $J$ and two maps $X:J \to \ko^\times, \ s: J \to \mathcal{S}$,
we can define a quiver $Q^J=(Q^J_0,Q^J_1)$ associated with $(J,X,s)$
as (i:vertices) $Q^J_0= J$, (ii:arrows) for $i,j \in J$, we put
$\mathtt{d}_{ij}$ many arrows from $i$ to $j$, where
$\mathtt{d}_{ij}$ is the order of the zero of
$d_{V_{s(i)},V_{s(j)}}(z)$ at $X(j)/X(i)$.

Then we obtain a symmetric Cartan matrix $\cm^J=(a^J_{ij})_{i,j \in J}$ associated with
$(J,X,s)$ by
\begin{align*}
a^J_{ij} = 2 \quad \text{ if } i =j \quad \text{ and } \quad a^J_{ij} = -\mathtt{d}_{ij}-\mathtt{d}_{ji} \quad \text{ if } i \ne j.
\end{align*}
Let $R^J$ be the quiver Hecke algebras associated with the symmetric Cartan matrix $\cm^J$ and the parameters (\cite{KL09,KL11,R08})
$$ \mathcal{Q}_{i,j}(u,v) =(u-v)^{\mathtt{d}_{ij}}(v-u)^{\mathtt{d}_{ji}} \ \ \text{ if } i \ne j \quad \text{ and } \quad \mathcal{Q}_{i,i}(u,v)=0
\ \ \text{ for all } i \in J.$$

\noindent
\textbf{Theorem} \cite{KKK13a}
 There exists a functor
$$\mathcal{F} : \Rep(R^J) \rightarrow \mathscr{C}_\g $$
where $\Rep(R^J)$ denotes the category of finite dimensional representations over $R^J$.
Moreover, the functor enjoys the following properties:
\begin{enumerate}
\item[({\rm a})] $\mathcal{F}$ is a tensor functor; that is, there exist $U_q'(\g)$-module isomorphisms
$$\mathcal{F}(R^J(0)) \simeq \ko \quad \text{ and } \quad \mathcal{F}(M_1 \circ M_2) \simeq \mathcal{F}(M_1) \tens \mathcal{F}(M_2)$$
for any $M_1, M_2 \in \Rep(R^J)$.
\item[({\rm b})] If the Cartan matrix $\cm^J$ is of type $A_n(n \ge 1)$, $D_n(n \ge 4)$, $E_6$, $E_7$ or $E_8$, then
the functor $\mathcal{F}$ is exact.
\end{enumerate}
\medskip

Thus the  generalized quantum affine Schur-Weyl duality functor provides the way of investigating the category $\mathscr{C}_\g $
via the category $\Rep(R^J)$ and the other way around (see \cite{KKKO14}).

Note that $\cm^J$ depends on the choice of $(J,X,s)$ and the denominators. Hence one may expect various exact functors
defined on $\Rep(R^J)$ for a fixed algebra $R^J$. In the forthcoming papers by the author and his collaborators
(\cite{KKKO14A,KKKO14D}), they will consider such situations, and the denominator formulas given in this paper will play
an important role.

The denominators of all normalized $R$-matrices $\Rnorm_{k,l}(z)$
for $A^{(1)}_{n}$, $C^{(1)}_{n}$ and $D^{(1)}_{n}$ were studied in \cite{AK,DO94,KKK13b} and
the denominators of normalized $R$-matrix $\Rnorm_{1,1}(z)$ (resp. $\Rnorm_{n,n}(z)$)
between vector representations (resp. spin representations) for all classical affine types are given in \cite{KMN2,Okado90}.
On the other hand, the explicit forms of normalized $R$-matrix $\Rnorm_{1,1}(z)$ for all classical affine types were studied in \cite{DO96,HYZ,J86,JMO00}.
With these results, we will compute the denominators $d_{k,l}(z)$ of all normalized $R$-matrices $\Rnorm_{k,l}(z)$ by employing
the framework given in \cite[Appendix A]{KKK13b}.

Our main results are
\begin{align*}
 d_{k,l}(z) =  \prod_{s=1}^{\min(k,l)} (z^t-(-q^t)^{|k-l|+2s})(z^t-(p^*)^t(-q^t)^{2s-k-l})
\end{align*}
if $V(\va_k)$ and $V(\va_l)$ are not spin representations, and
\begin{align*}
\begin{array}{ll}
   d_{k,n}(z) = \prod_{s=1}^{k}(z-(-1)^{n+k}q_s^{2n-2k-1+4s}) & \text{if } \g=B^{(1)}_{n} \text{ and } k < n, \\
   d_{k,n}(z)=\prod_{s=1}^{k}(z^2+(-q^{2})^{n-k+2s}) & \text{if } \g=D^{(2)}_{n+1} \text{ and } k < n. \end{array}
\end{align*}
Here,
$$t= \begin{cases} 2 & \text{ if } \g=D^{(2)}_{n+1}, \\ 1 & \text{otherwise}, \end{cases}, \ \ q_s^2=q \ \ \text{ and }
p^* \seteq (-1)^{\langle \rho^\vee ,\delta \rangle}q^{(\rho,\delta)}  $$
for the {\it null root} $\delta$ (see \eqref{eq: dual}).
Hence we can conclude that
\begin{enumerate}
\item[({\rm a})] $\Rnorm_{k,l}(z)$ of $A_{2n-1}^{(2)}$, $A_{2n}^{(2)}$ or $B_{n}^{(1)}$ has only simple poles,
\item[({\rm b})] $\Rnorm_{k,l}(z)$ of $D_{n+1}^{(2)}$ has a double pole at $z=\mqs^{s/2}$ if
$$ 2 \le k,l \le n-1, \ k+l > n, \ 2n+2-k-l \le s \le k+l \text{ and } s \equiv k+l \ {\rm mod} \ 2,$$
\item[({\rm c})] $\Rnorm_{k,l}(z)$ has a pole at $\pm(-q^t)^{\ell/t}$ only if $k \in \Z$ such that $2 \le \ell \le (\rho,\delta)$ (see \cite{FM}).
\end{enumerate}

This paper is organized as follows. In the first section, we recall
the notion of quantum affine algebras and $R$-matrices, briefly. In
the next section, we give the $U_q'(\g)$-module structure of the vector
representations and spin representations over $U_q'(\g)$. In
the third section, we study morphisms in
$\Hom_{U_q'(\g)}(V(\va_i)_a\tens V(\va_j)_b,V(\va_k)_c)$, called the {\it Dorey's} type morphisms. After that, we prove the existence of certain surjective homomorphisms which can be understood
as $D^{(2)}_{n+1}$-analogue of \cite[Lemma A.3.2]{KKK13b}. In the last section, we propose the general frame work for
computing the denominators, which is originated from \cite[Appendix A]{KKK13b}.
Then we compute $d_{1,n}(z)$ for $\g=D^{(2)}_{n+1}$ and the unknown denominators $d_{k,l}(z)$ of normalized $R$-matrices
for $\g=A^{(2)}_{2n-1}$, $A^{(2)}_{2n}$, $B^{(1)}_{n}$ and $D^{(2)}_{n+1}$,
by using the results in the previous sections. In the appendix, we provide a
table of $d_{k,l}(z)$ for all classical affine types for reader's
convenience.

\bigskip

\noindent
{\bf Acknowledgements.} The author would like to express his sincere gratitude to
Professor Seok-Jin Kang, Professor Masaki Kashiwara and Myungho Kim for many fruitful discussions.
The author gratefully acknowledge the hospitality of RIMS (Kyoto) during his visit in 2013 and 2014.

\section{Quantum affine algebras and $R$-matrices}
In this section, we briefly recall the backgrounds and theories on quantum affine algebras, their finite dimensional integral representations and $R$-matrices.
We refer to \cite{AK,KKK13a,Kas02} for precise statements and definitions.

\subsection{Quantum affine algebras and their representations.} Let $I = \{0,1, \ldots, n \}$ be a set of indices and set $I_0 \seteq I \setminus \{ 0 \}$. An {\it affine Cartan datum} is
a quadruple $(\cm,\wl,\Pi,\Pi^\vee)$ consisting of
\begin{enumerate}
\item[({\rm a})] a matrix $\cm$ of corank $1$, called the {\it affine Cartan matrix} satisfying
$$ ({\rm i}) \ a_{ii}=2 \ (i \in I), \quad ({\rm ii}) \ a_{ij} \in \Z_{\le 0}, \quad  ({\rm iii}) \ a_{ij}=0 \text{ if } a_{ji}=0$$
with $\mathsf{D}= {\rm diag}(\mathsf{d}_i \in \Z_{>0}  \mid i \in I)$ making $\mathsf{D}\cm$ symmetric,
\item[({\rm b})] a free abelian group $\wl$ of rank $n+2$, called the {\it weight lattice},
\item[({\rm c})] $\Pi = \{ \alpha_i \mid i\in I \} \subset \wl$, called the set of {\it simple roots},
\item[({\rm d})] $\Pi^{\vee} = \{ h_i \mid i\in I\} \subset \wl^{\vee} := \Hom( \wl, \Z )$, called the set of {\it simple coroots},
\end{enumerate}
which satisfy
\begin{itemize}
\item[(1)] $\langle h_i, \alpha_j \rangle  = a_{ij}$ for all $i,j\in I$,
\item[(2)] $\Pi$ and $\Pi^{\vee}$ are linearly independent sets,
\item[(3)] for each $i \in I$, there exists $\Lambda_i \in \wl$ such that $\langle h_i, \Lambda_j \rangle =\delta_{ij}$ for all $j \in I$.
\end{itemize}
We set $\rl = \bigoplus_{i \in I} \Z \alpha_i$, $\rl_+ = \bigoplus_{i \in I} \Z_{\ge 0} \alpha_i$, $\rl^\vee = \bigoplus_{i \in I} \Z h_i$ and
$\rl^\vee_+ = \bigoplus_{i \in I} \Z_{\ge 0} h_i$.
We choose the {\it imaginary root} $\delta=\sum_{i \in I}\mathsf{a}_i \alpha_i \in \rl_+$ and the {\it center} $c=\sum_{i \in I} \mathsf{c}_ih_i \in \rl^\vee_+$
such that (\cite[Chapter 4]{Kac})
$$ \{ \lambda \in \rl \ | \ \langle h_i,\lambda \rangle =0 \text{ for every } i \in I \} =\Z \delta  \text{ and }
\{ h \in \rl^\vee \ | \ \langle h,\alpha_i \rangle =0 \text{ for every } i \in I \} =\Z c.$$

Set $\h = \Q \tens_\Z \wl^\vee$.
Then there exists a symmetric bilinear form $( \ , \ )$ on $\h^*$ satisfying
$$ \langle h_i,\lambda \rangle = \dfrac{  2(\alpha_i,\lambda)}{(\alpha_i,\alpha_i)} \quad \text{ for any } i \in I \text{ and } \lambda \in \h^*.$$
We normalize the bilinear form by
$$ \langle c,\lambda \rangle = (\delta,\lambda) \quad \text{ for any } \lambda \in \h^*.$$

Let us denote by $\g$ the affine Kac-Moody Lie algebra associated with $(\cm,\wl,\Pi,\Pi^\vee)$ and by $\mathsf{W}$ the Weyl group of $\g$, generated by
$(\mathsf{s}_i)_{i \in I}$. We define $\g_0$ the subalgebra of $\g$ generated by the chevalley generators $e_i,f_i,$ and
$h_i$ for $i \in I_0$. Then $\g_0$ is a finite dimensional simple Lie algebra.

Let $\gamma$ be the smallest positive integer such that
$$ \gamma(\alpha_i,\alpha_i)/2 \in \Z \quad \text{ for any } i \in I.$$

Let $q$ be an indeterminate. For $m,n \in \Z_{\ge 0}$ and $i\in I$, we define
$q_i = q^{(\alpha_i,\alpha_i)/2}$ and
\begin{equation*}
 \begin{aligned}
 \ &[n]_i =\frac{ q^n_{i} - q^{-n}_{i} }{ q_{i} - q^{-1}_{i} },
 \ &[n]_i! = \prod^{n}_{k=1} [k]_i ,
 \ &\left[\begin{matrix}m \\ n\\ \end{matrix} \right]_i=  \frac{ [m]_i! }{[m-n]_i! [n]_i! }.
 \end{aligned}
\end{equation*}

\begin{definition} \label{Def: GKM}
The {\em quantum affine algebra} $U_q(\g)$ associated
with $(\cm,\wl,\Pi,\Pi^{\vee})$ is the associative
algebra over $\Q(q^{1/\gamma})$ with $1$ generated by $e_i,f_i$ $(i \in I)$ and
$q^{h}$ $(h \in \gamma^{-1}\wl^{\vee})$ satisfying following relations:
\begin{enumerate}
  \item  $q^0=1, q^{h} q^{h'}=q^{h+h'} $ for $ h,h' \in \gamma^{-1}\wl^{\vee},$
  \item  $q^{h}e_i q^{-h}= q^{\langle h, \alpha_i \rangle} e_i,
          \ q^{h}f_i q^{-h} = q^{-\langle h, \alpha_i \rangle }f_i$ for $h \in \gamma^{-1}\wl^{\vee}, i \in I$,
  \item  $e_if_j - f_je_i =  \delta_{ij} \dfrac{K_i -K^{-1}_i}{q_i- q^{-1}_i }, \ \ \text{ where } K_i=q_i^{ h_i},$
  \item  $\displaystyle \sum^{1-a_{ij}}_{k=0}
  (-1)^ke^{(1-a_{ij}-k)}_i e_j e^{(k)}_i =  \sum^{1-a_{ij}}_{k=0} (-1)^k
  f^{(1-a_{ij}-k)}_i f_jf^{(k)}_i=0 \quad \text{ for }  i \ne j, $
\end{enumerate}
where $e_i^{(k)}=e_i^k/[k]_i!$ and $f_i^{(k)}=f_i^k/[k]_i!$.
\end{definition}

We denote by $U_q'(\g)$ the subalgebra of $U_q(\g)$ generated by
$e_i,f_i,K^{\pm 1}_i$ $(i \in I)$ and we call it {\it also} the quantum
affine algebra. Throughout this paper, we mainly deal with
$U_q'(\g)$.

For $\uqpg$-modules $M$ and $N$, $M \otimes N$ becomes a $\uqpg$-module by the coproduct $\Delta$ of $\uqpg$:
$$ \Delta(q^h)=q^h \tens q^h, \ \ \Delta(e_i)=e_i \tens K_i^{-1}+1 \tens e_i, \ \Delta(f_i)=f_i \tens 1 +K_i \tens f_i.$$

Set $\wl_\cl \seteq \wl / \Z \delta$ and $\cl \colon \wl \to \wl_\cl$ as the canonical projection.

We say that a $\uqpg$-module $M$ is {\it integrable} provided that
\begin{enumerate}
\item[({\rm a})] $M$ decomposes into $\wl_\cl$-weight spaces; that is,
$$ M = \soplus_{ \mu \in \wl_\cl} M_\mu, $$
where $M_\mu \seteq \{ v \in M \ | \ K_i v = q^{\langle h_i, \mu \rangle } v \}$,
\item[({\rm b})] $e_i$ and $f_i$ $(i \in I)$ act on $M$ nilpotently.
\end{enumerate}

For $i \in I_0$, {\it the level $0$ fundamental weight} $\va_i$ is defined by
$$ \va_i \seteq {\rm gcd}(\mathsf{c}_0,\mathsf{c}_i)^{-1}(\mathsf{c}_0\Lambda_i-\mathsf{c}_i\Lambda_0) \in \wl.$$
Then $\{ \cl(\va_i) \ | \ i \in I_0 \}$ forms a basis for the space of {\it classical integral weight level $0$},
denoted by $\wl^0_\cl$, which is defined as follows:
$$ \wl^0_\cl = \{ \lambda \in \wl_\cl \ | \ \langle c,\lambda \rangle =0  \}.$$
The Weyl group $\mathsf{W}_0$ of $\g_0$, generated by $(\mathsf{s}_i)_{i \in I_0}$, acts on $\wl^0_\cl$ (see \cite[\S 1.2]{AK}).
We denote by $w_0$ the longest element of $\mathsf{W}_0$.

\begin{definition} \cite[\S 1.3]{AK} \label{def: fundamental module}
For $i \in I_0$, the {\it $i$th fundamental module} is a unique finite dimensional integrable $U_q'(\g)$-module $V(\va_i)$ satisfying the following properties:
\begin{enumerate}
\item[(1)] The weights of $V(\va_i)$ are contained in the convex hull of $\mathsf{W}_0\cl(\va_i)$.
\item[(2)] $V(\va_i)_{\cl(\va_i)}=\C(q) \mathsf{v}_{\va_i}$. (We call the vector $\mathsf{v}_{\va_i}$ a {\it dominant integral weight vector}.)
\item[(3)] For any $\mu \in \mathsf{W}_0\cl(\va_i)$, we can associate a non-zero vector $u_\mu$, called an {\it extremal vector of weight $\mu$}, such that
\begin{equation} \label{eq: extremal}
\mathcal{S}_i\cdot u_\mu \seteq  u_{\s_i\mu}= \begin{cases} f_i^{(\langle h_i, \mu \rangle)}u_\mu & \text{ if } \langle h_i, \mu \rangle \ge 0, \\
e_i^{(-\langle h_i, \mu \rangle)}u_\mu & \text{ if } \langle h_i, \mu \rangle \le 0, \end{cases} \quad \text{ for any $i \in I$.}
\end{equation}
\item[(4)] $\mathsf{v}_{\va_i}$ generates $V(\va_i)$ as a $\uqpg$-module.
\end{enumerate}
\end{definition}

Let $\ko$ be an algebraic closure of $\C(q)$ in $\cup_{m >0}\C((q^{1/m}))$. When we deal with $\uqpg$-modules, we regard the base field as $\ko$.

For a $\uqpg$-module $M$, we denote by ${}^*M$ the {\it right dual} and $M^*$ the {\it left dual} of $M$, if there exist $U_q'(\g)$-homomorphisms
$$ M^* \otimes M \overset{{\rm tr}}{\longrightarrow} \ko, \quad \ko \longrightarrow  M \otimes M^*
\quad \text{and} \quad  M \otimes {}^*M \overset{{\rm tr}}{\longrightarrow} \ko, \quad \ko \longrightarrow  M{}^* \otimes M.$$

Recall that $V(\va_i)$ is finite dimensional and has the right dual and left dual as follows:
\begin{align} \label{eq: dual}
V(\varpi_i)^*\seteq  V(\varpi_{i^*})_{(p^*)^{-1}}, \ {}^*V(\varpi_i)\seteq  V(\varpi_{i^*})_{p^*} \ \  \text{ and } \ \
p^* \seteq (-1)^{\langle \rho^\vee ,\delta \rangle}q^{(\rho,\delta)}
\end{align}
where
\begin{itemize}
\item $^*$ is the involution of $I_0$ defined by the image of $\va_i$ under the action $w_0$; i.e., $w_0(\va_i)=-\va_{i^*}$,
\item $\rho$ is defined by  $\langle h_i,\rho \rangle=1$ and $\rho^\vee$ is defined by $\langle \rho^\vee,\alpha_i  \rangle=1$ for all $i \in I$.
\end{itemize}

\medskip

We call an integrable $U_q'(\g)$-module $M$ {\it good} if  $M$
satisfies certain properties. In this paper, the whole definition of the good module is not needed. Thus we refer
\cite{Kas02} for the precise definition of good module. However, we
would like to emphasize one condition of the good module:

\smallskip

A good module $M$ contains the unique (up to constant) weight vector $v_M$ of weight $\lambda$, such that
$$   {\rm wt}(M) \subset \lambda + \sum_{i \in I_0} \Z_{\ge 0} \cl(\alpha_i).$$
We call $v_M$ the {\it dominant extremal weight vector} and $\lambda$ {\it dominant extremal weight}.
For instance, the $i$th fundamental representation
is a good module.

\medskip

For an indeterminate $z$
and a $\uqpg$-module $M$, let us denote by $M_z= \{ u_z \ | \ u \in M \}$
the $\uqpg$-module $\ko[z^{\pm 1}]\otimes M$ with the action of $\uqpg$ given by
$$e_i(u_z)=z^{\delta_{i,0}}(e_iu)_z, \quad
f_i(u_z)=z^{-\delta_{i,0}}(f_iu)_z, \quad
K_i(u_z)=(K_iu)_z.
$$

{\it We sometimes use the notation $u$ for $u_z$ to simplify equations.} (For example, see the proof of Proposition \ref{prop: akl})

\subsection{ Normalized and universal $R$-matrices.}
We call a $\ko[z^{\pm 1}] \otimes \uqpg$-module homomorphism between $M \otimes N_z$ and $N_z \otimes M$ as an {\it intertwiner}.
It is known that, for finite dimensional integral $\uqpg$-modules $M$ and $N$,
there exists an intertwiner
$$\Runiv_{M,N}(z): M \otimes N_z \to N_z \otimes M$$
which satisfies
\begin{align}\label{eq: Runiv property}
\Runiv_{M,N\otimes N'}(z)=(N_z \otimes \Runiv_{M,N'}(z)) \circ (\Runiv_{M,N}(z)\otimes N'_z).
\end{align}
We call $\Runiv_{M,N}$ the {\it universal $R$-matrix} \cite{FR92}.

\begin{definition}
For good modules $M$ and $N$, the {\it normalized $R$-matrix} $\Rnorm_{M,N}$ is the $\uqpg$-module
homomorphism
\begin{equation} \label{Def: R(MN)}
\Rnorm_{M,N}: M_{z_M} \otimes N_{z_N} \to \ko(z_M,z_N) \otimes_{\ko[z_M^{\pm 1},z_N^{\pm 1}]}
 N_{z_N} \otimes M_{z_M}
\end{equation}
which satisfies
$$ \Rnorm_{M,N} \circ z_M = z_M \circ \Rnorm_{M,N}, \ \Rnorm_{M,N} \circ z_N = z_N \circ \Rnorm_{M,N} \
\text{ and } \ \Rnorm_{M,N}(\mathsf{v}_{M}\otimes \mathsf{v}_{N})=\mathsf{v}_{N} \otimes \mathsf{v}_{M}.$$
\end{definition}

\cite[Corollary 2.5]{AK} tells that, for good modules $M$ and $N$
$$\Hom_{\ko(z)\otimes\uqpg}( M \otimes N_z, N_z \otimes M) \simeq \ko(z),$$
and hence there exists $a_{M,N}(z) \in \ko(z)$ such that
\begin{equation}\label{Def: a(MN)}
\Runiv_{M,N}(z)=a_{M,N}(z)\Rnorm_{M,N}(z).
\end{equation}

Note that $$\Rnorm_{M,N}(z)(M \otimes N_z) \subset \ko(z) \otimes_{\ko[z^{\pm 1}]}(N_z \otimes M) $$
and there exists a unique monic polynomial $d_{M,N}(z) \in \ko[z]$ such that
\begin{equation}\label{Def: d(MN)}
d_{M,N}(z)\Rnorm_{M,N}(z)(M \otimes N_z) \subset (N_z \otimes M).
\end{equation}
We call $d_{M,N}(u)$ the {\it denominator} of $\Rnorm_{M,N}(z)$.

\begin{lemma}[{\cite[Lemma C.15]{AK}}] \label{lem:dvw avw}
Let $V', V'', V$ and $W$ be irreducible $U_q'(\g)$-modules. Assume that we have a surjective $U_q'(\g)$-homomorphism
$$ V'\otimes V'' \twoheadrightarrow V.$$ Then we have
  \begin{align*}
   \dfrac{ d_{W,V'}(z) d_{W,V''}(z) a_{W,V}(z)}{d_{W,V}(z)a_{W,V'}(z)a_{W,V''}(z)}
   \quad \text{and} &\quad
   \dfrac{ d_{V',W}(z) d_{V'',W}(z) a_{V,W}(z)}{d_{V,W}(z)a_{V',W}(z)a_{V'',W}(z)} \in \ko[z^{\pm 1}].
  \end{align*}
\end{lemma}

\section{Vector and spin representation.} \label{Append: vector spin reps}
In this section, we record the $\uqpg$-module structure of
\begin{itemize}
\item $V(\va_1)$, called the {\it vector representation},
\item $V(\varpi_n)$, called the {\it spin representation}, for $\g=B^{(1)}_{n}$ or $D^{(2)}_{n+1}$.
\end{itemize}
As a vector space, the vector representation can be expressed as follows (\cite[Chapter 11]{HK02}):
$$V(\va_1)=\big( \soplus_{j=1}^n \ko v_j  \big) \oplus \big( \soplus_{j=1}^n \ko v_{\overline{j}} \big) \oplus W$$
where
\begin{center}
\begin{tabular}{ | c | c | c | c | c  | } \hline
$\g$ & $A^{(2)}_{2n-1}$ &  $B^{(1)}_{n}$ &  $A^{(2)}_{2n}$ &  $D^{(2)}_{n+1}$ \\ \hline
$W$ & $\emptyset$ & $\ko v_0$  & $\ko v_\emptyset$ & $\ko v_0 \oplus \ko v_\emptyset$ \\ \hline
\end{tabular}
\end{center}

The actions of $e_i$, $f_i$ and $q^h$ are defined by follows:
$$q^h \cdot v_j=q^{\langle h,{\rm wt}(v_j)\rangle}v_j \quad \text{for} \quad h \in \wl^\vee_\cl,$$
\begin{center}
\fontsize{10}{10}\selectfont
\begin{tabular}{ | c | c | c  | } \hline
$\g$ & $e_i$ &  $f_i$ \\ \hline
$A_{2n-1}^{(2)}$ & $e_i v_j =\begin{cases}
    v_i & \text{if $j= i+1$ and $i \neq n$, }\\
    v_{\overline{i+1}} & \text{if $j= \overline{i}$ and $i \neq n$, }\\
    v_n & \text{if $j= \overline{n}$ and $i = n$,}\\
         v_{\overline{2}} & \text{if $j= 1$ and $i = 0$,}\\
         v_{\overline{1}} & \text{if $j= 2$ and $i = 0$,}\\
    0 & \text{otherwise,} \end{cases}$ &
    $f_i v_j =\begin{cases}
    v_{i+1} & \text{if $j= i$ and $i \neq n$,}\\
    v_{\overline{i}} & \text{if $j= \overline{i+1}$ and $i \neq n$,} \\
    v_{\overline{n}} & \text{if $j= n$ and $i = n$,}\\
    v_{1} & \text{if $j= \overline{2}$ and $i = 0$,}\\
    v_{2} & \text{if $j= \overline{1}$ and $i = 0$,}\\
    0 & \text{otherwise},
  \end{cases}$\\ \hline
$A_{2n}^{(2)}$ & $ e_i v_j =\begin{cases}
    v_i & \text{if $j= i+1$ and $i \neq n$, }\\
    v_{\overline{i+1}} & \text{if $j= \overline{i}$ and $i \neq n$, }\\
    v_n & \text{if $j= \overline{n}$ and $i = n$,}\\
         v_{\emptyset} & \text{if $j= 1$ and $i = 0$,}\\
         [2]_{0}v_{\overline{1}} & \text{if $j= \emptyset$ and $i = 0$,}\\
    0 & \text{otherwise},
  \end{cases} $ & $f_i v_j =\begin{cases}
    v_{i+1} & \text{if $j= i$ and $i \neq n$,}\\
    v_{\overline{i}} & \text{if $j= \overline{i+1}$ and $i \neq n$,} \\
    v_{\overline{n}} & \text{if $j= n$ and $i = n$,}\\
    v_{\emptyset} & \text{if $j= \overline{1}$ and $i = 0$,}\\
    [2]_{0}v_{1} & \text{if $j= \emptyset$ and $i = 0$,}\\
    0 & \text{otherwise},
  \end{cases}$ \allowdisplaybreaks  \\ \hline
  $B_{n}^{(1)}$ &  $ e_i v_j =\begin{cases}
    v_i & \text{if $j= i+1$ and $i \neq n$, }\\
    v_{\overline{i+1}} & \text{if $j= \overline{i}$ and $i \neq n$, }\\
    v_0 & \text{if $j= \overline{n}$ and $i = n$,}\\
    [2]_{n}v_n & \text{if $j= 0$ and $i = n$,}\\
         v_{\overline{2}} & \text{if $j= 1$ and $i = 0$,}\\
         v_{\overline{1}} & \text{if $j= 2$ and $i = 0$,}\\
    0 & \text{otherwise},
  \end{cases}$ & $f_i v_j =\begin{cases}
    v_{i+1} & \text{if $j= i$ and $i \neq n$,}\\
    v_{\overline{i}} & \text{if $j= \overline{i+1}$ and $i \neq n$,} \\
    v_{0} & \text{if $j= n$ and $i = n$,}\\
    [2]_{n}v_{\overline{n}} & \text{if $j= 0$ and $i = n$,}\\
    v_{1} & \text{if $j= \overline{2}$ and $i = 0$,}\\
    v_{2} & \text{if $j= \overline{1}$ and $i = 0$,}\\
    0 & \text{otherwise},
  \end{cases}$\\ \hline
$D_{n+1}^{(2)}$ & $ e_i v_j =\begin{cases}
    v_i & \text{if $j= i+1$ and $i \neq n$, }\\
    v_{\overline{i+1}} & \text{if $j= \overline{i}$ and $i \neq n$, }\\
    v_0 & \text{if $j= \overline{n}$ and $i = n$,}\\
    [2]_nv_n & \text{if $j= 0$ and $i = n$,}\\
        v_{\emptyset} & \text{if $j= 1$ and $i = 0$,}\\
         [2]_0v_{\overline{1}} & \text{if $j= \emptyset$ and $i = 0$,}\\
    0 & \text{otherwise},
  \end{cases}$ & $f_i v_j =\begin{cases}
    v_{i+1} & \text{if $j= i$ and $i \neq n$,}\\
    v_{\overline{i}} & \text{if $j= \overline{i+1}$ and $i \neq n$,} \\
    v_{0} & \text{if $j= n$ and $i = n$,}\\
    [2]_nv_{\overline{n}} & \text{if $j= 0$ and $i = n$,}\\
    v_{\emptyset} & \text{if $j= \overline{1}$ and $i = 0$,}\\
    [2]_0v_{1} & \text{if $j= \emptyset$ and $i = 0$,}\\
    0 & \text{otherwise},
  \end{cases}$  \\ \hline
\end{tabular}
\fontsize{11}{11}\selectfont
\end{center}
where $${\rm wt}(v_j)=\epsilon_j, \ {\rm wt}(v_{\overline{j}})=-\epsilon_j \text{ for } j=1,\ldots,n \text{ and } {\rm wt}(v_0)={\rm wt}(v_\emptyset)=0.$$

For $\g=B_{n}^{(1)}$ or $\g=D_{n+1}^{(2)}$, the spin representation $V(\va_n)$ is the $\ko$-vector space with a basis
\begin{align*}
  \mathsf{B}_{{\rm sp}} = \{ (m_1, \ldots, m_n) \, ; \, m_i = + \, \text{or} \, - \}.
\end{align*}
Its $U_q'(\g)$-module structure is given by defining the action of $e_i$, $f_i$ and $q^h$ as follows:
$$q^{h} v = q^{\langle h , {\rm wt}(v)\rangle} v   \quad \text{for $h\in \wl_\cl^\vee$, where
  ${\rm wt}(v) = \frac{1}{2} \sum_{k=1}^n m_k \epsilon_k$,} $$
\begin{center}
\fontsize{10}{10}\selectfont
\begin{tabular}{ | c | c | } \hline
$\g$ & $e_i$,  $f_i$ \\ \hline
  $B_{n}^{(1)}$ &  $e_iv = \begin{cases}
    (m_1,\ldots, \overset{i}{+}, \overset{i+1}{-}, \ldots, m_n)
    & \text{if $i\not=,n$ and $m_i=- $, $m_{i+1}=+$,} \\
    (m_1,\ldots, m_{n-1}, \overset{n}{+})
    &\text{if $i=n$ and $m_{n}=-$,} \\
    (-,-, m_3,\ldots, m_n)
    & \text{if $i=0$ and $m_1=m_2=+$,}\\
    0 & \text{otherwise},
           \end{cases}$ \\ \cline {2-2}
    & $f_iv = \begin{cases}
    (m_1,\ldots, \overset{i}{-}, \overset{i+1}{+}, \ldots, m_n)
    &\text{if $i\not=n$ and $m_i=+$, $m_{i+1}=-$,} \\
    (m_1,\ldots, m_{n-1}, \overset{n}{-})
    & \text{if $i=n$ and $m_{n}=+$,} \\
    (+,+, m_3,\ldots, m_n)
    & \text{if $i=0$ and $m_1=m_2=-$,}\\
    0 & \text{otherwise},
           \end{cases}$ \allowdisplaybreaks \\ \hline
 $D^{(2)}_{n+1}$ & $e_iv = \begin{cases}
    (m_1,\ldots, \overset{i}{+}, \overset{i+1}{-}, \ldots, m_n)
    &\text{if $i\not=,n$ and $m_i=- $, $m_{i+1}=+$,} \\
    (m_1,\ldots, m_{n-1}, \overset{n}{+})
    &\text{if $i=n$ and $m_{n}=-$,} \\
    (-, m_2,\ldots, m_n)
    & \text{if $i=0$ and $m_1=+$,}\\
    0 & \text{otherwise},
           \end{cases}$ \\ \cline {2-2}
 & $f_iv= \begin{cases}
    (m_1,\ldots, \overset{i}{-}, \overset{i+1}{+}, \ldots, m_n)
    &\text{if $i\not=n$ and $m_i=+$, $m_{i+1}=-$,} \\
    (m_1,\ldots, m_{n-1}, \overset{n}{-})
    & \text{if $i=n$ and $m_{n}=+$,} \\
    (+, m_2,\ldots, m_n)
    & \text{if $i=0$ and $m_1=-$,}\\
    0 & \text{otherwise}.
           \end{cases}$ \\ \hline
\end{tabular}
\fontsize{11}{11}\selectfont
\end{center}

\section{Surjective homomorphisms between integrable $\uqpg$-modules } \label{sec: sujective homo}
In this section, we first study the morphisms in
 $${\rm Hom}_{U_q'(\g)}\big( V(\varpi_i)_a
\tens V(\varpi_j)_b, V(\varpi_k)_c \big) \quad \text{ for } i,j,k\in I_0 \text{ and } a,b,c \in \ko^\times.$$
These kinds of morphisms are known as {\it Dorey's type morphisms} and have been investigated in \cite{CP96} for
the classical untwisted affine types $A^{(1)}_{n}$, $B^{(1)}_{n}$,
$C^{(1)}_{n}$ and $D^{(1)}_{n}$. In the last part of this section,
we study the surjective homomorphisms which can be understood
as $D^{(2)}_{n+1}$-analogue of the surjective homomorphisms given in
\cite[(A.17)]{KKK13b}
\medskip

Hereafter, we will use the following convention frequently:
\begin{center}
For a statement $P$, $\delta(P)$ is $1$ if $P$ is true and $0$ if $P$ is false.
\end{center}

\medskip

By the result on $B^{(1)}_n$ in \cite{CP96}, {\it it suffices to consider when $\g=A^{(2)}_{2n-1}$ $(n\ge 3)$, $A^{(2)}_{2n}$ $(n\ge 2)$ and $D^{(2)}_{n+1}$} $(n\ge 2)$.

\medskip

The finite Dynkin diagrams of $\g_0$ associated with $\g$ are given as follows:
$$
C_n\colon\xymatrix@R=3ex{ *{\circ}<3pt> \ar@{-}[r]_<{1 \
}^{\epsilon_1-\epsilon_2 \qquad \ } & {} \ar@{.}[r] & *{\circ}<3pt>
\ar@{-}[r]_>{\,\,\,\ n-1}^{ \qquad \  \epsilon_{n-1}-\epsilon_{n} }
&*{\circ}<3pt>\ar@{=>}[r]_>{\,\,\,\,n}^{ \qquad \ \  \epsilon_n }
&*{\circ}<3pt> } \ (A^{(2)}_{2n-1}, \ A^{(2)}_{2n}) \quad B_n\colon\xymatrix@R=3ex{ *{\circ}<3pt>
\ar@{-}[r]_<{1 \ }^{\epsilon_1-\epsilon_2 \qquad \ }   & {}
\ar@{.}[r] & *{\circ}<3pt> \ar@{-}[r]_>{\,\,\,\ n-1}^{ \qquad \
\epsilon_{n-1}-\epsilon_{n} }
&*{\circ}<3pt>\ar@{<=}[r]_>{\,\,\,\,n}^{ \qquad \ 2\epsilon_n }
&*{\circ}<3pt> } \ (D^{(2)}_{n+1}).
$$
We denote by $V_0(\va_i)$ for $i \in I_0$, the highest weight $U_q(\g_0)$-module with the highest weight $\va_i$.
\medskip

Throughout this paper, we set
\begin{equation}\label{def: t}
t= \begin{cases} 2 & \text{ if } \g=D^{(2)}_{n+1}, \\ 1 & \text{ otherwise,} \end{cases} \quad \text{ and } \quad
\vartheta= \begin{cases} 1 & \text{ if } \g=B^{(1)}_{n} \text{ or } D^{(2)}_{n+1}, \\ 0 & \text{ otherwise.} \end{cases}
\end{equation}

\subsection{$i+j=k \le n-\vartheta$.}\label{subsec: i+j=k}

Recall that there exists an injective $U_q(\g_0)$-module homomorphism (see. \cite[Chapter 8]{HK02})
$$ \Phi_{i,j}:V_0(\va_{i+j}) \rightarrowtail V_0(\va_{i}) \otimes V_0(\va_{j})  \quad \text{ for } \quad i+j \le n -\vartheta$$
given by
\begin{align} \label{map: injection g_0}
u_\lambda \longmapsto v_\lambda= \sum_{\lambda=\mu+\xi}
C_{\mu,\xi}^{\lambda} u_\mu \otimes u_\xi \qquad (C_{\mu,\xi}^{\lambda} \in \ko)
\end{align}
where $\lambda \in \mathsf{W}_0\cdot\va_{i+j}$ and $\mu$ (resp. $\xi$) runs
over the elements in $\mathsf{W}_0\cdot\va_{i}$ (resp. $\mathsf{W}_0\cdot\va_{j}$).

For a positive integer $l \le n -\vartheta$, we sometimes write
$\lambda \in {\rm wt}(V_0(\va_{l})) $ as a sequence $(\lambda_1,
\ldots, \lambda_n) \in \{ 1,0,-1 \}^{n}$ such that
$\lambda=\sum_{k=1}^{n} \lambda_k \epsilon_k$.

In \eqref{map: injection g_0}, since $\Phi_{i,j}$ is a $U_q(\g_0)$-homomorphism and $V(\va_{i+j})$ is generated by
$u_{\va_{i+j}}$, we can observe that
\begin{equation} \label{eq: multi 0}
\text{$\lambda_k \ge 0$ implies that $\mu_k, \xi_k \ge 0$ and
$\lambda_k \le 0$ implies that $\mu_k, \xi_k \le 0$.}
\end{equation}
Since $\lambda_k \in \{ 1,0,-1\}$, we can conclude that
\begin{align} \label{eq: mu_k xi_k=0}
\mu_k \xi_k=0 \quad \text{ for all } \quad 1 \le k \le n.
\end{align}
From the observation \eqref{eq: multi 0},  $C_{\mu,\xi}^{\lambda}$ must be the same as $C_{\s_k\mu,\s_k\xi}^{\s_k\lambda}$ whenever
$\langle h_k ,\lambda \rangle \neq 0$.

\begin{proposition}\label{prop: c}
Set
\begin{equation} \label{eq: coeff c}
\begin{aligned}c_{\mu,\xi}^{\lambda} & = \# \{ (a,b) \ | \ a<b, \
(\mu_a,\xi_a)=(0,1), \ \mu_b \neq 0 \} \\
& \qquad \qquad \qquad + \# \{ (a,b) \ | \ a<b, \
(\mu_a,\xi_a)=(-1,0), \ \xi_b \ne 0 \}.
\end{aligned}
\end{equation}
Then the $C_{\mu,\xi}^{\lambda}$ in \eqref{map: injection g_0} is given as follows:
$$C_{\mu,\xi}^{\lambda} =(-q_1)^{c_{\mu,\xi}^{\lambda}}.$$
\end{proposition}
\begin{proof}
First, we check that
$c_{\s_k\mu,\s_k\xi}^{\s_k\lambda}=c_{\mu,\xi}^{\lambda}$ whenever
$\langle h_k,\lambda \rangle \neq 0$. To do this, it suffices to
consider $(a,b)=(k,k+1)$. Then one can easily check that
\begin{align*}& \# \{ (a,b) \ | \ a<b, \
(\mu_a,\xi_a)=(0,1), \ \mu_b\ne 0 \} \\ & \qquad\qquad + \# \{ (a,b)
| \
a<b, \ (\mu_a,\xi_a)=(-1,0), \xi_b\ne 0 \} \\
&= \# \{ (a,b) \ | \ a<b, \ ((\s_k\mu)_a,(\s_k\xi)_a)=(0,1), \
(\s_k\mu)_b\ne 0 \} \\ & \qquad\qquad + \# \{ (a,b) | \ a<b, \
((\s_k\mu)_a,(\s_k\mu)_a)=(-1,0), (\s_k\xi)_b\ne 0 \}.
\end{align*}
Thus we can assume that $\lambda=\va_{i+j}$.  If $k  \ge i+j$, then
$e_k v_\lambda= 0$, trivially.  When $1 \le k < i+j$,
\begin{align*}
0=e_k v_\lambda & = \sum_{\substack{ (\mu_k,\mu_{k+1})=(0,1) \\
(\xi_k,\xi_{k+1})=(1,0)}} C_{\mu,\xi}^{\lambda}q_1^{-1} v_{\s_k\mu}
\otimes v_{\xi} + \sum_{\substack{ (\mu_k,\mu_{k+1})=(1,0) \\
(\xi_k,\xi_{k+1})=(0,1)}} C_{\mu,\xi}^{\lambda} v_{\mu} \otimes
v_{\s_k\xi}.
\end{align*}
Thus, for $(\mu_k,\mu_{k+1})=(0,1) $ and $(\xi_k,\xi_{k+1})=(1,0)$,
we have
$$ c_{\mu,\xi}^{\lambda} = c_{\s_k\mu,\s_k\xi}^{\lambda} +1 $$
which implies our assertion.
\end{proof}

Now we shall
determine $x,y \in \ko^\times$ such that there exists an injective $U_q'(\g)$-module homomorphism
\begin{align} \label{eq:coeff x,y}
 V(\va_{i+j}) \rightarrowtail V(\va_{i})_x \otimes V(\va_{j})_y.
 \end{align}
 The strategy in this subsection can be explained as follows:
As a $U_q(\g_0)$-module, we have an injection
$$V_0(\va_{i+j}) \rightarrowtail V(\va_{i})_{x} \otimes V(\va_{j})_{y}.$$
By using characterization of $V(\va_{i+j})$ given in \cite[$\S$ 1.3]{AK}, it suffices to determine $x$ and $y$ satisfying
$$ C^{\s_0 \lambda}_{\s_0 \mu,\s_0 \xi} = f(x)g(y) C^{\lambda}_{\mu,\xi}$$
where
\begin{itemize}
\item $\lambda$, $\mu$ and $\xi$ are extremal weights and $\langle h_0,\lambda \rangle \ne 0$,
\item $f(x)$ and $g(y)$ arise from the action of $e_0$ or $f_0$ on $V(\va_i)_x$ or $V(\va_j)_y$.
\end{itemize}
\medskip

Recall the notion $\mathcal{S}_k \cdot u_{\mu}=u_{\s_k\mu}$ for an extremal weight $\mu$ in \eqref{eq: extremal}.

\begin{proposition} Let $\g=A^{(2)}_{2n-1}$ $(n \ge 3)$. Then the $x$, $y$ in $\eqref{eq:coeff x,y}$ are given as follows:
$$x = (-q)^{j} \quad \text{ and } \quad y = (-q)^{-i}.$$
\end{proposition}
\begin{proof}
The Dynkin diagram of $A^{(2)}_{2n-1}$ is given as follows:
$$
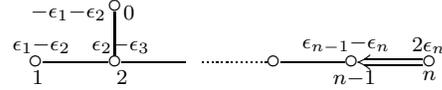
\xymatrix@R=3ex{ & *{\circ}<3pt> \ar@{-}[d]_<{\qquad \qquad -\epsilon_1-\epsilon_2  }^<{0}\\
*{\circ}<3pt> \ar@{-}[r]_<{1 \ }^{\epsilon_1-\epsilon_2 \qquad \ }
&*{\circ}<3pt> \ar@{-}[r]_<{2}^{\epsilon_2-\epsilon_3 \qquad \ }  &
{} \ar@{.}[r] & *{\circ}<3pt> \ar@{-}[r]_>{\,\,\,\ n-1}^{ \qquad \
\epsilon_{n-1}-\epsilon_{n} }
&*{\circ}<3pt>\ar@{<=}[r]_>{\,\,\,\,n}^{ \qquad \ \  2\epsilon_n }
&*{\circ }<3pt> }.$$
It suffices to consider $\lambda \in \mathsf{W}_0 \cdot \va_{i+j}$ such that
 $\lambda_1,\lambda_2 \ge 0$. Thus it is enough to consider when $\mu_1,\mu_2,\xi_1,\xi_2 \ge 0$. Then we have
$$ \mathcal{S}_0 \cdot v_\lambda=v_{\s_0 \lambda} = \sum C^{\lambda}_{\mu,\xi} \
x^{\delta(\mu_1=1)+\delta(\mu_2=1)}y^{\delta(\xi_1=1)+\delta(\xi_2=1)}
\ v_{\s_0 \mu} \otimes v_{\s_0 \xi}.$$
Here
$\s_0(\varepsilon_1,\varepsilon_2,\varepsilon_3,\ldots,\varepsilon_n)=
\s_0(-\varepsilon_2,-\varepsilon_1,\varepsilon_3,\ldots,\varepsilon_n)$ ($\varepsilon_i \in \{ -1,0,1\}$).
Thus
$$ C^{\s_0\lambda}_{\s_0\mu,\s_0\xi} = x^{\delta(\mu_1=1)+\delta(\mu_2=1)}y^{\delta(\xi_1=1)+\delta(\xi_2=1)} C^{\lambda}_{\mu,\xi}.$$
On the other hand, by \eqref{eq: coeff c},
\begin{align*}
c^{\s_0\lambda}_{\s_0\mu,\s_0\xi}-c^{\lambda}_{\mu,\xi}   &=
+\delta(\mu_1=1) \times \# \{ b>1 \ | \
\xi_b \ne 0 \} + \delta(\mu_2=1) \times  \# \{ b>2 \ | \ \xi_b \ne 0 \} \\
 & \quad - \delta(\xi_1=1) \times  \# \{ b>1 \ | \ \mu_b \ne 0 \}  - \delta(\xi_2=1)\times \# \{ b>2 \ | \ \mu_b \ne 0
 \} \\
 &= \delta(\mu_1=1) \times j + \delta(\mu_2=1) \times
 (j-\delta(\xi_2=1)) \\ & \quad - \delta(\xi_1=1) \times i -\delta(\xi_2=1) \times
 (i-\delta(\mu_2=1)).
\end{align*}
By \eqref{eq: mu_k xi_k=0}, $\mu_i\xi_i=0$ ($i=1,2$) and hence we can conclude that
$$ c^{\s_0\lambda}_{\s_0\mu,\s_0\xi} -c^{\lambda}_{\mu,\xi} =
-(\delta(\xi_1=1)+\delta(\xi_2=1))\times
i+(\delta(\mu_1=1)+\delta(\mu_2=1))\times j.$$ Thus we have
$$x = (-q)^{j} \quad \text{ and } \quad y = (-q)^{-i}.$$
\end{proof}

\begin{proposition} Let $\g=A^{(2)}_{2n}$ $(n \ge 2)$. Then the $x$, $y$ in $\eqref{eq:coeff x,y}$ are given as follows:
$$x = (-q)^{j} \quad \text{ and } \quad y = (-q)^{-i}.$$
\end{proposition}

\begin{proof}
The Dynkin diagram of $A^{(2)}_{2n}$ is given as follows:
$$
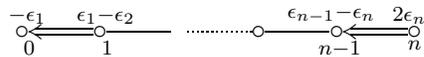
\xymatrix@R=3ex{ *{\circ}<3pt> \ar@{<=}[r]_<{0}^{ \ -\epsilon_1
\qquad \ \  } &*{\circ}<3pt> \ar@{-}[r]_<{1}^{\epsilon_1-\epsilon_2
\qquad \ }   & {} \ar@{.}[r] & *{\circ}<3pt> \ar@{-}[r]_>{\,\,\,\
n-1}^{ \qquad \ \epsilon_{n-1}-\epsilon_{n} }
&*{\circ}<3pt>\ar@{<=}[r]_>{\,\,\,\,n}^{ \qquad \  2\epsilon_{n} }
&*{\circ}<3pt> }
$$
It suffices to consider $\lambda \in \mathsf{W}_0 \cdot \va_{i+j}$ such that
$\langle h_0, \lambda \rangle <0$ and hence $\lambda_1=1$. Then we have
$$ \mathcal{S}_0 \cdot v_\lambda=v_{\s_0 \lambda} = e_0^{(2)}v_{\lambda} = C^{\lambda}_{\mu,\xi}
x^{\delta(\mu_1=1)}y^{\delta(\xi_1=1)} \ v_{\s_0 \mu} \otimes v_{\s_0\xi}.$$
Here $\s_0(\varepsilon_1,\varepsilon_2,\ldots,\varepsilon_n)=
\s_0(-\varepsilon_1,\varepsilon_2,\ldots,\varepsilon_n)$. Thus
$$ C^{\s_0\lambda}_{\s_0\mu,\s_0\xi} = x^{\delta(\mu_1=1)}y^{\delta(\xi_1=1)} C^{\lambda}_{\mu,\xi}.$$
On the other hand, by \eqref{eq: coeff c},
$$ C^{\s_0\lambda}_{\s_0\mu,\s_0\xi}  = (-q)^{\delta(\mu_1=1) \times \# \{ b>1 \mid \xi_b \ne 0  \}
} (-q)^{-\delta(\xi_1=1) \times \# \{ b>1 \mid \mu_b \ne 0  \} }
C^{\lambda}_{\mu,\xi}.$$ Thus we can conclude that
$$x = (-q)^{j} \quad \text{ and } \quad y = (-q)^{-i}.$$
\end{proof}

\begin{proposition} \label{prop: mor D 1}
Let $\g=D^{(2)}_{n+1}$. Then the $x$, $y$ in $\eqref{eq:coeff x,y}$ are given as follows:
$$x = \mqs^{j/2} \quad \text{ and } \quad y = \mqs^{-i/2}.$$
\end{proposition}

\begin{proof}
The Dynkin diagram of $D^{(2)}_{n+1}$ is given as follows:
$$
\xymatrix@R=3ex{ *{\circ}<3pt> \ar@{<=}[r]_<{0}^{  -\epsilon_1
\qquad \ }  &*{\circ}<3pt> \ar@{-}[r]_<{1}^{\epsilon_1-\epsilon_2
\qquad \ }    & {} \ar@{.}[r] & *{\circ}<3pt> \ar@{-}[r]_>{\,\,\,\
n-1}^{ \qquad \ \epsilon_{n-1}-\epsilon_{n} }
&*{\circ}<3pt>\ar@{=>}[r]_>{\,\,\,\,n}^{ \qquad \  \epsilon_{n} }
&*{\circ}<3pt> }
$$
It suffices to consider $\lambda \in \mathsf{W}_0 \cdot \va_{i+j}$ such that
$\langle h_0, \lambda \rangle <0$. Thus we assume that $\lambda_1=1$. Note that $q_1=q^2$. Then
we have
\begin{align*}
\mathcal{S}_0 \cdot v_\lambda=v_{\s_0 \lambda}  = e_0^{(2)} v_\lambda = \sum C^{\lambda}_{\mu,\xi}
x^{2\delta(\mu_1=1)}y^{2\delta(\xi_1=1)} \ v_{\s_0 \mu} \otimes
v_{\s_0\xi}.
\end{align*}
Here $\s_0(\varepsilon_1,\varepsilon_2,\ldots,\varepsilon_n)=
\s_0(-\varepsilon_1,\varepsilon_2,\ldots,\varepsilon_n)$. Thus
$$ C^{\s_0\lambda}_{\s_0\mu,\s_0\xi} = x^{2\delta(\mu_1=1)}y^{2\delta(\xi_1=1)} C^{\lambda}_{\mu,\xi}.$$
On the other hand, by \eqref{eq: coeff c},
$$ C^{\s_0\lambda}_{\s_0\mu,\s_0\xi}  = \mqs^{\delta(\mu_1=1) \times \# \{ b>1 \mid \xi_b \ne 0  \}
} \mqs^{-\delta(\xi_1=1) \times \# \{ b>1 \mid \mu_b \ne 0  \} }
C^{\lambda}_{\mu,\xi}.$$ Thus we can conclude that
$$x^2 = \mqs^{j} \quad \text{ and } \quad y^2 = \mqs^{-i},$$
which yield our assertion.
\end{proof}

\begin{theorem} \label{thm: pij iota ij 1}
For $i+j=k \le n-\vartheta$, there exists a surjective $U_q'(\g)$-module homomorphism
\begin{align} \label{eq: p ij}
  p_{i,j} \colon V(\varpi_i)_{(-q^t)^{-j/t}} \otimes V(\varpi_j)_{(-q^t)^{i/t}}
  \twoheadrightarrow V(\varpi_{k}).
  \end{align}
By taking dual, there exists an injective $U_q'(\g)$-module homomorphism
\begin{align}\label{eq: iota ij}
  \iota_{i,j} \colon V(\varpi_k)  \rightarrowtail V(\varpi_i)_{(-q^t)^{j/t}}
  \otimes V(\varpi_j)_{(-q^t)^{-i/t}}.
  \end{align}
\end{theorem}
\begin{proof}The proof immediately comes from the previous propositions.
\end{proof}

\subsection{$i=j=n, \ k <n$ for $\g=D^{(2)}_{n+1}$.} In this subsection, we fix $\g$ as $D^{(2)}_{n+1}$.
Recall that there exists an injective $U_q(B_n)$-module homomorphism (see. \cite[Chapter 8]{HK02})
$$ V_0(\va_{i}) \rightarrowtail V_0(\va_{n}) \otimes V_0(\va_{n})$$
given by
\begin{align} \label{map: injection g_0 2}
u_\lambda \longmapsto v_\lambda= \sum_{\lambda=\mu+\xi}
C_{\mu,\xi}^{\lambda} u_\mu \otimes u_\xi
\end{align}
where $\lambda \in \mathsf{W}_0\cdot\va_{i}$ and $\mu,\xi\in \mathsf{W}_0\cdot\va_{n}$.

We sometimes write $\mu \in {\rm wt}(V_0(\va_{n})) $ as a sequence
$(\mu_1, \ldots, \mu_n) \in \{ +,- \}^{n}$ such that
$$\mu=\sum_{k=1}^{n} \dfrac{\mu_k}{2} \epsilon_k.$$

\begin{proposition} \label{prop: c spin}
Set
\begin{equation} \label{eq: coeff c 2}
\begin{aligned}
{}_1c_{\mu,\xi}^{\lambda} & = \# \{ (a,b) \ | \ a<b, \
(\mu_a,\xi_a)=(-,+), \ (\mu_b,\xi_b)=(+,-) \}, \\
{}_2c_{\mu,\xi}^{\lambda} & = \# \{ a \ | \ (\mu_a,\xi_a)=(-,+) \}, \\
\varphi(c) & = (-q)^{c} \mqs^{\frac{c(c-1)}{2}}.
\end{aligned}
\end{equation} Then $C_{\mu,\xi}^{\lambda}$ in \eqref{map: injection g_0 2} is given as follows:
$$C_{\mu,\xi}^{\lambda} =\mqs^{{}_1c_{\mu,\xi}^{\lambda}} \varphi({}_2c_{\mu,\xi}^{\lambda}).$$
\end{proposition}

\begin{proof} As in Proposition \ref{prop: c}, one can check that $C_{\s_k\mu,\s_k\xi}^{\s_k\lambda}=C_{\mu,\xi}^{\lambda}$ whenever
$\langle h_k,\lambda \rangle \neq 0$. Thus we can assume that
$\lambda=\va_{i}$. If $k \le i$, then $e_k u_\lambda=0$, trivially.
Thus, for $k > i$, we have
\begin{align*}
0= e_k v_\lambda  = \begin{cases} \sum_{\substack{ (\mu_{k},\mu_{k+1}) =(-,+) \\
(\xi_{k},\xi_{k+1}) =(+,-) }} C_{\mu,\xi}^{\lambda}(q^2)^{-1} u_{\s_k \mu}
\otimes u_{\xi} \\
 \qquad + \sum_{\substack{ (\mu_{k},\mu_{k+1}) =(+,-) \\
(\xi_{k},\xi_{k+1}) =(-,+) }}C_{\mu,\xi}^{\lambda}u_{\mu} \otimes
u_{\s_k \xi} \quad  \text{ if }  i< k <n, \\
\sum_{ (\mu_{n},\xi_n) =(-,+) } C_{\mu,\xi}^{\lambda}q^{-1}  u_{\s_n
\mu} \otimes u_{\xi}\\  \qquad   +\sum_{ (\mu_{n},\xi_n) =(+,-) }
u_{\mu} \otimes u_{\s_n \xi} \qquad \quad \ \  \text{ if }  k=n.
\end{cases}
\end{align*}
Thus we have
$$ C_{\mu,\xi}^{\lambda} =
\begin{cases}
-q^2 C_{\s_k\mu,\s_k\xi}^{\lambda}  & \text{ if } i < k <n \text{ and } (\mu_{k},\xi_{k}) =(-,+),  \\[1ex]
(-q)^{-1} C_{\s_n\mu,\s_n\xi}^{\lambda}  & \text{ if } k=n \text{
and } \mu_{n}=+. \end{cases}
$$
On the other hand, for $i< k <n$ and $(\mu_{k},\xi_{k}) =(-,+)$, we
have \begin{align*} & {}_1c_{\mu,\xi}^{\lambda} = \begin{cases}
{}_1c_{\s_k\mu,\s_k\xi}^{\lambda}-1 & \text{ if } i < k <n \text{ and } (\mu_{k},\xi_{k}) =(-,+),  \\
{}_1c_{\mu,\xi}^{\lambda} =
{}_1c_{\s_k\mu,\s_k\xi}^{\lambda}+{}_2c_{\mu,\xi}^{\lambda} & \text{
if } k=n \text{ and } \mu_{n}=+, \end{cases} \\
& {}_2c_{\mu,\xi}^{\lambda}=
\begin{cases}
{}_2c_{\s_k\mu,\s_k\xi}^{\lambda}  & \text{ if } i < k <n \text{ and } (\mu_{k},\xi_{k}) =(-,+),  \\
  {}_2c_{\mu,\xi}^{\lambda}={}_2c_{\s_k\mu,\s_k\xi}^{\lambda}-1 &
\text{ if } k=n \text{ and } \mu_{n}=+,\end{cases}\end{align*} which
yield our assertion.
\end{proof}

\begin{theorem} \label{thm: nnk}
For $k \le n-1$, there exists a surjective $U_q'(D^{(2)}_{n+1})$-module homomorphism
\begin{align} \label{eq: p ij 2}
  p_{n,k} \colon V(\varpi_n)_{\pm \sqrt{-1} \mqs^{-\frac{n-k}{2}}} \otimes V(\varpi_n)_{\mp \sqrt{-1} \mqs^{\frac{n-k}{2}}}
  \twoheadrightarrow V(\varpi_{k}).
  \end{align}
By taking dual, there exists an injective $U_q'(D^{(2)}_{n+1})$-module homomorphism
\begin{align}\label{eq: iota ij 2}
  \iota_{n,k} \colon V(\varpi_k)  \rightarrowtail V(\varpi_n)_{\pm \sqrt{-1} \mqs^{\frac{n-k}{2}}}
  \otimes V(\varpi_n)_{\mp \sqrt{-1} \mqs^{-\frac{n-k}{2}}}.
  \end{align}
\end{theorem}

\begin{proof}
We apply the same strategy of \S \ref{subsec: i+j=k}; i.e., we determine the $x$ and $y$ in \eqref{eq:coeff x,y}. As in Proposition
\ref{prop: mor D 1}, we first consider $\lambda \in \mathsf{W}_0 \cdot \va_k$ with $\lambda_1=1$ and hence $\mu_1=\xi_1=+$. In this
 case, we have
\begin{align*}
\mathcal{S}_0\cdot v_\lambda=v_{\s_0\lambda} =e_0^{(2)} v_\lambda = \sum C^{\lambda}_{\mu,\xi}
xy \ u_{\s_0 \mu} \otimes u_{\s_0\xi}.
\end{align*}
On the other hand
$${}_1c_{\mu,\xi}^{\lambda}= {}_1c_{\s_0\mu,\s_0\xi}^{\s_0\lambda}, \quad {}_2c_{\mu,\xi}^{\lambda} = {}_2c_{\s_0\mu,\s_0\xi}^{\s_0\lambda}.$$
 Thus we conclude that $$xy=1.$$
 Consider $\lambda \in \mathsf{W}_0 \cdot \va_{i}$ with $\langle h_0, \lambda \rangle
 =0$. Equivalently $\lambda_1=0$ and hence $-\mu_1=\xi_1$. In this case,
\begin{align*}
0=e_0v_\lambda & =\sum_{(\mu_1,\xi_1)=(+,-)}C_{\mu,\xi}^{\lambda}
q^{-1} x u_{\s_0\mu} \otimes u_\xi +
\sum_{(\mu_1,\xi_1)=(-,+)}C_{\mu,\xi}^{\lambda} y u_\mu \otimes
u_{\s_0\xi}.
\end{align*}
Thus, for $\mu_1=+$, we have
$$ C_{\s_0\mu,\s_0\xi}^{\lambda}=C_{\mu,\xi}^{\lambda} (-q)^{-1} \times \dfrac{x}{y} =C_{\mu,\xi}^{\lambda} (-q)^{-1} \times x^2.$$
On the other hand,
$${}_1c^{\lambda}_{\s_0\mu,\s_0\xi} =  {}_1c^{\lambda}_{\mu,\xi} + \# \{ b > 1 \mid (\mu_b,\xi_b)=(+,-) \} \ \text{ and } \
{}_2c^{\lambda}_{\s_0\mu,\s_0\xi} = {}_2c^{\lambda}_{\mu,\xi}+1.$$
Thus we have
$$ x^2=-\mqs^{n-k},$$
which yields our assertion.
\end{proof}

\subsection{$j=1$ and $i=k=n$ for $\g=A^{(2)}_{2n}$.} In this subsection, we show that there exists a surjective $U_q'(A^{(2)}_{2n})$-homomorphism
\begin{align} \label{eq: 1nn}
 V(\va_{n})_{(-q)^{-1}} \otimes V(\va_{1})_{(-q)^{n}} \twoheadrightarrow V(\va_{n}).
\end{align}
Indeed, we do not use \eqref{eq: 1nn} in this paper. However, for the forthcoming works, we present the existence of such a homomorphism.

Similar to the previous subsections, we determine the relations among $a$, $b$ and $c$ such that
\begin{align} \label{eq: 1nn before}
V(\va_{n})_a \rightarrowtail V(\va_{1})_b \otimes V(\va_{n})_c.
\end{align}
Recall that for $k \in I_0$ (see \cite[Table 1]{OS08}),
$$ V(\va_k) \simeq \soplus_{j=0}^{k} V_0(\va_j) \text{ as a $U_q(C_n)$-module.} $$
Here $V_0(\va_0)$ is the trivial $\uqpg$-module $\ko$. Thus
$$ \text{$V_0(\va_n)^{\oplus 2} \rightarrowtail V(\va_{1}) \otimes V(\va_{n})$ as a $U_q(C_n)$-module.}$$

The crystal graph of $V(\va_1)$ is given by (see \cite[Example 11.1.4]{HK02})
$$\xymatrix@C=7ex@R=0.5ex{
&1\ar[r]^-{1}&2\ar[r]^-{2}&\cdots\cdots\ar[r]^-{n-2}
&n-1\ar[r]^-{n-1}&n \ar[dd]^-{n}\\
\emptyset\ar[ur]^-{0}&&\\
&\overline{1}\ar[ul]^-{0}&
\overline{2}\ar[l]^-{1}&\cdots\cdots\cdots\ar[l]^-{2}&\overline{n-1}\ar[l]^-{n-2}
&\overline{n}\ar[l]^-{n-1}
}
$$
We denote by $$  \text{$\mathsf{u}$ the dominant integral weight vector of $V(\va_n)$ with its weight $\va_n=\sum_{i \in I_0} \epsilon_i$.}$$

For $i_1,\ldots,i_k,j_1,\ldots,j_l \in I_0$, we set
$\mathsf{u}[\overline{i_1},\ldots,\overline{i_l},\widehat{j_1},\ldots,\widehat{j_l}]$ the vector in $V_0(\va_{n-l})$ with
its weight given by
$$ {\rm wt}(\mathsf{u}[\overline{i_1},\ldots,\overline{i_k},\widehat{j_1},\ldots,\widehat{j_l}])={\rm wt}(\mathsf{u})-\sum_{s=1}^{k}2\epsilon_{i_s}-\sum_{t=1}^{l}\epsilon_{j_t},$$
if such a weight vector exists in $V_0(\va_{n-l})$.

The map \eqref{eq: 1nn before}, if it exists, sends $\mathsf{u}$ to the following vector, say $\mathsf{v}$:
$$ \mathsf{u} \longmapsto \mathsf{v} =
v_{\emptyset} \otimes \mathsf{u} + (-qb^{-1}c)\left( \sum_{k=1}^n
(-q)^{k-1} v_k \otimes \mathsf{u}[\widehat{k}]\right),$$ which is
unique (up to constant) in the sense that it satisfies $e_i
\mathsf{v}=0$ for $i \in I_0$, and $f_0 u=0$.

In $V(\va_{n})_a$, we have
\begin{align}\label{eq: coeff a}
 \mathcal{S}_0 \cdot \mathsf{u} = a \mathcal{S}_w \cdot \mathsf{u} \quad \text{ where } \mathcal{S}_w=\mathcal{S}_1\mathcal{S}_2\cdots\mathcal{S}_n
\text{ for }  w=\mathsf{s}_1\mathsf{s}_2\cdots \mathsf{s}_n \in \mathsf{W}_0.
\end{align}
On the other hand,
\begin{align*}
\mathcal{S}_0 \cdot \mathsf{v} & = e_0^{(2)}\mathsf{v}= c \ v_{\overline{1}} \otimes \mathsf{u}[\widehat{1}]
-q c \ v_{\emptyset} \otimes \mathsf{u}[\overline{1}] - qb^{-1}cd \sum_{k \ne 1} (-q)^{k-1} v_k \otimes \mathsf{u}[\overline{1},\widehat{k}], \allowdisplaybreaks \\
\mathcal{S}_w \cdot \mathsf{v} & = f_1^{(2)}f_2^{(2)}\cdots f_{n-1}^{(2)}f_n\mathsf{v} = v_\emptyset \otimes \mathsf{u}[\overline{1}]
+ (-qb^{-1}c) \left( \sum_{k \ne n} (-q)^{k-1} v_{k+1} \otimes \mathsf{u}[\overline{1},\widehat{k+1}] \right) \\ & \hspace{34ex}
+(-qb^{-1}c)(-q)^{n-1} v_{\overline{1}} \otimes \mathsf{u}[\widehat{1}],
\end{align*}
where $d$ is an element in $\ko^\times$ such that
\begin{align} \label{eq: d}
 e_0^{(2)} \mathsf{u}[\widehat{k}]=d \times \mathsf{u}[\overline{1},\widehat{k}] \quad \text{ for $k \ne 1$ in } V(\va_n)_c.
\end{align}
By \eqref{eq: coeff a}, we can conclude that
\begin{align} \label{eq: abc}a=-qc, \ \ b=a(-q)^n, \ \ d=c.\end{align}
Now, it suffices to show that $d=c=1$.

\begin{proposition} For $1 \ne k \in I_0$, the coefficient $d$ in \eqref{eq: d} must be equal to $1$; i.e.,
$$e_0^{(2)} \mathsf{u}[\widehat{k}]=\mathsf{u}[\overline{1},\widehat{k}] \quad \text{ in } V(\va_n)_c.$$
\end{proposition}

\begin{proof}
By Definition \ref{Def: GKM} (3), we have
$$f_1e_0 \mathsf{u}[\widehat{2}]= e_0f_1\mathsf{u}[\widehat{2}]= e_0 \mathsf{u}[\widehat{1}]=[2]_0\mathsf{u}[\overline{1}].$$
Thus
$$e_1e_0\mathsf{u}[\widehat{2}]=[2]_0e_1^{(2)}\mathsf{u}[\overline{1}]=[2]_0\mathsf{u}[\overline{2}].$$
From the actions $e_i$ ($i \in I$) on $V(\va_n)_c$, we have
\begin{align}\label{eq: step 1 d}
e_0e_1e_0^{(2)}\mathsf{u}[\widehat{2}]=ce_0e_1 \mathsf{u}[\overline{1},\widehat{2}]=ce_0\mathsf{u}[\overline{2},\widehat{1}]
=c[2]_0\mathsf{u}[\overline{1},\overline{2}].
\end{align}
Since all vectors in $V(\va_n)$ are annihilated by the action $e_0^{(3)}$, the relation in Definition \ref{Def: GKM} (4) implies that
\begin{align}\label{eq: step 2 d}
e_0e_1e_0^{(2)}\mathsf{u}[\widehat{2}] = (e_1e_0^{(3)}+e_0^{(2)}e_1e_0-e_0^{(3)}e_1)\mathsf{u}[\widehat{2}]
=e_0^{(2)}e_1e_0\mathsf{u}[\widehat{2}] =
[2]_0e_0^{(2)}\mathsf{u}[\overline{2}]
=[2]_0u[\overline{1},\overline{2}].
\end{align}
From \eqref{eq: abc}, \eqref{eq: step 1 d} and \eqref{eq: step 2 d}, we can conclude that
$$ d=c=1.$$
\end{proof}

Now, we have the following theorem.

\begin{theorem}There exists a surjective $U_q'(A^{(2)}_{2n})$-module homomorphism
\begin{align} \label{eq: p 1nn}
  p_{1,n} \colon V(\varpi_n)_{(-q)^{-1}} \otimes V(\varpi_1)_{(-q)^{n}}
  \twoheadrightarrow V(\varpi_{n}).
  \end{align}
By taking dual, there exists an injective $U_q'(A^{(2)}_{2n})$-module homomorphism
\begin{align}\label{eq: iota 1nn}
  \iota_{1,n} \colon V(\varpi_n)  \rightarrowtail V(\varpi_1)_{(-q)^{n}}
  \otimes V(\varpi_n)_{(-q)^{-1}}.
  \end{align}
\end{theorem}

\subsection{$D^{(2)}_{n+1}$-analogue of the surjective homomorphisms given in \cite[(A.17)]{KKK13b}} This subsection is devoted to prove the following
lemma.

\begin{lemma} \label{lemma: k,l to n,n}
Let  $\eta, \eta' \in  \{ \sqrt{-1},-\sqrt{-1} \}$ and $1 \le k,l
\leq n-1$ such that $k+ l = n$. Then there exists a surjective
$U_q'(D^{(2)}_{n+1})$-module homomorphism
\begin{align*}
    V(\varpi_k)_{\eta\mqs^{-\frac{l}{2}}} \otimes V(\varpi_l)_{\eta'\mqs^{\frac{k}{2}}}
    \twoheadrightarrow V(\varpi_{n})_{-1} \otimes V(\varpi_{n}).
\end{align*}
\end{lemma}

\begin{proof}
Note that $\eta/\eta' = \pm 1$. By Theorem \ref{thm: nnk}, there are two injective $U_q'(D^{(2)}_{n+1})$-homomorphisms
\begin{align*}
&  \psi_1\colon  V(\varpi_{k})_{\eta\mqs^{-\frac{l}{2}}}
  \rightarrowtail V(\varpi_{n})_{-1} \otimes V(\varpi_{n})_{\mqs^{\frac{-n+k-l}{2}}}, \\[1ex]
&   \psi_2\colon V(\varpi_{l})_{\eta'\mqs^{\frac{k}{2}}}
   \rightarrowtail V(\varpi_{n})_{-\mqs^{\frac{n+k-l}{2}}}\otimes V(\varpi_{n}),
\end{align*}
by taking dual. Then we can obtain
$\varphi= ({\rm id}_{V(\varpi_{n})_{-1}} \otimes {\rm tr} \otimes {\rm id}_{V(\varpi_{n})_{-1} \otimes}) \circ(\psi_1\tens \psi_2)$,
$$ \varphi: V(\varpi_{k})_{\eta\mqs^{-\frac{l}{2}}} \otimes V(\varpi_{l})_{\eta'\mqs^{\frac{k}{2}}}
\longrightarrow V(\varpi_{n})_{-1} \otimes V(\varpi_{n}),$$
since $V(\varpi_{n})_{\mqs^{\frac{-n+k-l}{2}}}$ and $V(\varpi_{n})_{-\mqs^{\frac{n+k-l}{2}}}$ are dual to each other.

Applying the argument of \cite[Lemma A.3.2]{KKK13b}, we have
$$\varphi(v\tens w)
\equiv{\rm tr}(u_{-\va_{n}}\tens u_{\va_{n}})v_1\tens w_1 \mod
\soplus_{\lambda\not= -\va_k +\va_{n}}\big(V(\va_n)_{-1}\big)_\lambda\tens
V(\va_n)_{-\va_k+\va_l-\lambda},$$
where
\begin{itemize}
\item $v$ is the $U_q(B_n)$-lowest weight vector of $V(\va_k)_{\eta\mqs^{-\frac{l}{2}}}$ of weight $-\va_k$,
\item $w$ is the $U_q(B_n)$-highest weight vector of $V(\va_l)_{\eta'\mqs^{\frac{k}{2}}}$ of weight $\va_l$,
\item $v_1$ is a non-zero vector of $V(\va_n)_{-1}$ of weight $-\va_k+\va_{n}$,
\item $w_1$ is a non-zero vector of $V(\va_n)$ of weight $\va_l-\va_{n}$.
\end{itemize}
Thus $\varphi$ is non-zero. Then our assertion follows from the fact that $V(\va_n)_{-1}\tens V(\va_n)$ is irreducible.
\end{proof}

\section{The computation of denominators between fundamental representations}
For simplicity, we write $\Rnorm_{k,l}$ for $\Rnorm_{V(\va_k),V(\va_l)}$ in \eqref{Def: R(MN)}, $d_{k,l}$ for $d_{V(\va_k),V(\va_l)}$ in \eqref{Def: d(MN)}
and  $a_{k,l}$ for $a_{V(\va_k),V(\va_l)}$ in \eqref{Def: a(MN)}.

By the result of \cite[Appendix A]{AK} and \cite{Chari}, the denominator $d_{k,l}(z)$ and the element $a_{k,l}(z) \in \ko(z)$ are symmetric
with respect to the indices $k$ and $l$; that is,
\begin{align} \label{eq: symmetric}
 d_{k,l}(z)=d_{l,k}(z) \quad\text{and}\quad a_{k,l}(z)=a_{l,k}(z).
\end{align}

\subsection{General framework.} In this subsection, we propose the strategy for computing $d_{k,l}(z)$, which is originated from
\cite[Appendix A]{KKK13b}.

Note that we have a surjective homomorphism
\begin{equation} \label{p(k-1,1)}
 p_{l-1,1} \colon V(\varpi_{l-1})_{(-q^t)^{-1/t}} \otimes V(\varpi_1)_{(-q^t)^{l-1/t}} \twoheadrightarrow V(\varpi_l)
\quad \text{ if }  l \le n-\vartheta,
\end{equation}
by the previous section.

\begin{assumption}\label{Assum: framework} Assume the followings:
\begin{itemize}
\item[(A)] We know $a_{k,l'}(z)$ for $k \in I_0$ and $l' \le l-1$.
\item[(B)] We know $d_{1,1}(z)$ for all $\g$, and $d_{1,n}(z)$ for $\g=B^{(1)}_{n}$ or $\g=D^{(2)}_{n+1}$.
\end{itemize}
\end{assumption}

With these assumptions and \eqref{eq: Runiv property}, consider the following commutative diagram:
\begin{align} \label{Diagram: Runiv}
    \xymatrix@R=6.5ex{
    V(\varpi_k) \otimes V(\varpi_{l-1})_{(-q^t)^{-1/t}z} \otimes V(\varpi_{1})_{(-q^t)^{l-1/t}z}
    \ar[rr]^{\qquad \qquad \quad V(\varpi_k)\otimes p_{l-1,1}} \ar[d]_{\Runiv_{k,l-1}((-q^t)^{-1/t}z) \otimes V(\varpi_{1})_{(-q^t)^{l-1/t}z}} &&
    V(\varpi_{k}) \otimes V(\varpi_{l})_{z} \ar[dd]_{\Runiv_{k,l}(z)} \\
    V(\varpi_{l-1})_{(-q^t)^{-1/2}z}  \otimes V(\varpi_{k}) \otimes V(\varpi_{1})_{(-q^t)^{l-1/2}z}
    \ar[d]_{V(\varpi_{l-1})_{(-q^t)^{-1/t}z} \otimes \Runiv_{k,1}((-q^t)^{l-1/t}z)} &\\
    V(\varpi_{l-1})_{(-q^t)^{-1/t}z} \otimes V(\varpi_{1})_{(-q^t)^{l-1/t}z}  \otimes V(\varpi_{k})\ar[rr]^{\qquad \qquad \qquad p_{l-1,1}\otimes V(\varpi_k) } &&
    V(\varpi_{l})_{z} \otimes V(\varpi_{k}).
    }
\end{align}
Then we have
\begin{align}\label{Diagram: Runiv2}
   \xymatrix@R=3.5ex{
    v_{[1,\ldots,k]} \otimes v_{[1,\ldots, l-1]} \otimes v_{l}
    \ar@{|-_{>}}[r] \ar@{|-_{>}}[d] &
    v_{[1,\ldots,k]} \otimes v_{[1,\ldots, l-1,l]} \ar@{|-_{>}}[dd]\\
    a_{k,l-1}((-q^t)^{-1/t}z) v_{[1,\ldots, l-1]} \otimes v_{[1,\ldots,k]} \otimes v_{l}
    \ar@{|-_{>}}[d] &\\
*++{a_{k,l-1}((-q^t)^{-1/t}z) a_{k,1}((-q^t)^{l-1/t}z)
         v_{[1,\ldots, l-1]} \otimes w} \ar@{|-_{>}}[r] &
   a_{k,l}(z) v_{[1,\ldots, l-1,l]} \otimes v_{[1,\ldots,k]},
    }
\end{align}
where \begin{itemize}
\item $v_{[1,\ldots, a]}$ is the dominant extremal weight vector of $V(\va_{a})$ for $a \in I_0$,
\item $w=\Rnorm_{k,1}((-q^t)^{l-1/t}z)( v_{[1,\dots,k]} \otimes v_l)$.
\end{itemize}
By observing the vector $w$, we can get an equation explaining the relationship between
$$a_{k,l-1}(-q^{-1}z) a_{k,1}((-q)^{l-1}z) \quad \text{ and } \quad a_{k,l}(z).$$
By Assumption \ref{Assum: framework} (A), we can compute $a_{k,l}(z)$ by using an induction.

After getting $a_{k,l}(z)$, we use the formulas in Lemma \ref{lem:dvw avw},
by applying two surjective homomorphisms in Section \ref{sec: sujective homo}
\begin{align}
&   p_{k-1,1} \colon V(\varpi_{k-1})_{(-q^t)^{-1/t}} \otimes V(\varpi_{1})_{(-q^t)^{k-1/t}}   \twoheadrightarrow V(\varpi_{k}), \label{eq:first surj} \\
&   p^*_{k-1,1} \colon V(\varpi_k)_{(-q^t)^{-1/t}} \otimes V(\varpi_1)_{(p^*)(-q^t)^{-k/t}} \rightarrow V(\varpi_{k-1}), \label{eq:second surj}
\end{align}
and setting $W=V(\va_l)$ or $V(\va_n)$, to get two elements in $\ko[z^{\pm 1}]$ which are described in terms of $d_{k,l}(z)$'s and $a_{k,l}(z)$'s.
Here \eqref{eq:second surj} is the composition of $U_q'(\g)$-homomorphisms given as follows:
\begin{align*}
V(\varpi_k)_{(-q^t)^{-1/t}} \otimes V(\varpi_1)_{(p^*)(-q^t)^{-k/t}} & \hookrightarrow
V(\varpi_{k-1}) \otimes V(\va_1)_{(-q^t)^{-k/t}} \otimes  V(\varpi_1)_{(p^*)(-q^t)^{-k/t}} \\
& \hspace{-8ex} \twoheadrightarrow V(\varpi_{k-1}) \otimes \ko \simeq V(\varpi_{k-1}).
\end{align*}

Since we know the forms of $a_{k,l}(z)$'s, two elements in $\ko[z^{\pm 1}]$ can be described in terms of
$d_{k,l}(z)$'s and polynomials in $\ko[z]$ (up to constant multiple of $\ko[z^{\pm 1}]^\times$).

By the assumptions, we know $d_{1,1(z)}$, $d_{1,n}(z)$ and hence we can
compute $d_{k,l}(z)$ and $d_{k,n}(z)$, by manipulating the two elements in $\ko[z^{\pm 1}]$ and using inductions.

\bigskip

The denominator $d_{1,1}(z)$ of $\Rnorm_{1,1}(z): V(\va_1) \otimes V(\va_1)_z \to V(\va_1)_z \otimes V(\va_1)$ are computed in \cite{KMN2}
(see also \cite{HYZ} for $\g=A^{(2)}_{2}$) as follows:
\begin{align} \label{eq: denominator d11}
  d_{1,1}(z)=(z^t-(q^2)^t)(z^t-(p^*)^t).
\end{align}

The denominator $d_{1,n}(z)$ of $\Rnorm_{1,n}(z): V(\va_1) \otimes V(\va_n)_z \to V(\va_n)_z \otimes V(\va_1)$ for $\g=B^{(1)}_{n}$
is computed in \cite{DO96} as follows:
\begin{align}\label{eq:deno1n B}
d_{1,n}(z)=d_{n,1}(z)= z-(-1)^{n+1}q_s^{2n+1}.
\end{align}

Considering Assumption \eqref{Assum: framework}, the only missing part is the denominator $d_{1,n}(z)$ for $\g=D^{(2)}_{n+1}$.

\subsection{The denominator $d_{1,n}(z)$ for $\g=D^{(2)}_{n+1}$.} To compute the denominator $d_{1,n}(z)$ for $\g=D^{(2)}_{n+1}$,
we follow the notations and arguments given in \cite[Section 4]{KMN2}.

\medskip

By the $U_q'(D^{(2)}_{n+1})$-module structure of $V(\va_1)$ and $V(\va_n)$ in Section \ref{Append: vector spin reps}, we have
$$V(\va_1) \simeq V_0(\varpi_1) \oplus V_0(0)\text{ and }  V(\va_n) \simeq
V_0(\va_n) \text{ as $U_q(B_n)$-modules.}$$
Here $V_0(\va_n)$ (resp. $V_0(0)$) is the highest $U_q(B_n)$-module with the highest weight $\va_n$ (resp. $0$).
Thus we have
$$V(\va_n) \otimes V(\va_1) \simeq V_0(\lambda) \oplus V_0(\varpi_n)^{\oplus 2}
\text{ as a $U_q(B_n)$-module}, $$
where $\lambda=(\frac{3}{2},\frac{1}{2},\ldots,\frac{1}{2})$.
Let
$$m_n^+=(+,\ldots,+) \quad \text{and} \quad  m^i=(+,\ldots,+, \overset{i}{-},+,\ldots,+) \quad (1 \le i \le n)$$ be the elements in
$V(\varpi_n)$. Then we have the following lemmas by the direct
calculation:

\begin{lemma} \label{Lem: h.w.v} Let
$u_\lambda$, $u_{\varpi_n}^1$ and $u_{\varpi_n}^2$ be the
$U_q(B_n)$-highest weight vectors with the weight $\lambda$,
$\varpi_n$ and $\varpi_n$ in $V(\va_n)_x \otimes V(\va_1)_y$ respectively. Then we have
\begin{enumerate}
\item[{\rm (a)}] $u_{\lambda}=(m_n^+) \otimes v_1$,
\item[{\rm (b)}] $ u_{\varpi_n}^1=  [2]_0^{-1}  (m_n^+) \otimes v_\emptyset$,
\item[{\rm (c)}] $u_{\varpi_n}^2=\sum_{k=1}^n (-1)^k q^{2k} (m^{n+1-k}) \otimes v_{n+1-k} + [2]_n^{-1} (m_n^+) \otimes v_0$.
\end{enumerate}
\end{lemma}

\begin{lemma} \label{Lem: h.w.v_2} Let
$\tilde u_\lambda$, $ \tilde u_{\varpi_n}^1$ and $ \tilde u_{\varpi_n}^2$ be the $U_q(B_n)$-highest weight vectors with the weight
$\lambda$, $\varpi_n$ and $\varpi_n$ in $V(\va_1)_y \otimes V(\va_n)_x$,  respectively. Then we have
\begin{enumerate}
\item[{\rm (a)}] $\tilde u_{\lambda}=(1) \otimes (m_n^+) $,
\item[{\rm (b)}] $\tilde u_{\varpi_n}^1=  [2]_0^{-1}  v_\emptyset \otimes (m_n^+)$,
\item[{\rm (c)}] $\tilde u_{\varpi_n}^2=\sum_{k=1}^n (-1)^{n+1-k} q^{-2(n+1-k)} v_k  \otimes (m^{k}) + q^{-1} [2]_n^{-1}  v_0 \otimes (m_n^+) $.
\end{enumerate}
\end{lemma}

Hence $\Rnorm_{1,n} : V(\va_1)_y \otimes V(\va_n)_x
\to V(\va_n)_x \otimes V(\va_1)_y $  can be expressed by
\begin{align*}
\Rnorm_{1,n}( \tilde u_{\lambda}  )=  u_{\lambda} \quad \text{ and } \quad \Rnorm_{1,n}(\tilde u^i_{\varpi_n}  )=\sum_{j=1}^2  a_{ji}^{\varpi_n}
u_{\varpi_n}^j.
\end{align*}

The following lemmas can be obtained by direct calculations.

\begin{lemma} For the highest weight vectors defined in {\rm Lemma \ref{Lem: h.w.v}}, we have
\begin{enumerate}
\item[{\rm (a)}] $f_0(u_{\varpi_n}^1)=x^{-1}y^{-1}\left( q^{-1}x \right)u_{\lambda}$,
\item[{\rm (b)}] $f_0(u_{\varpi_n}^2)=x^{-1}y^{-1}\left( (-1)^nq^{2n}y \right)u_{\lambda}$,
\item[{\rm (c)}] $e_1\cdots e_{n-1}e_n^{(2)}e_{n-1} \cdots e_2e_1e_0(u_{\varpi_n}^1)=(y)u_{\lambda}$,
\item[{\rm (d)}] $e_1\cdots e_{n-1}e_n^{(2)}e_{n-1} \cdots e_2e_1e_0(u_{\varpi_n}^2)=(q^{-1}x)u_{\lambda}$,
\end{enumerate}
in $V(\va_n)_x \otimes V(\va_1)_y$.
\end{lemma}

\begin{lemma} For the highest weight vectors defined in {\rm Lemma \ref{Lem: h.w.v_2}}, we have
\begin{enumerate}
\item[{\rm (a)}] $f_0(\tilde u_{\varpi_n}^1)=x^{-1}y^{-1}(x) \tilde u_{\lambda}$,
\item[{\rm (b)}] $f_0(\tilde u_{\varpi_n}^2)= x^{-1}y^{-1}((-1)^n q^{-2n-2} y) \tilde u_{\lambda}$,
\item[{\rm (c)}] $e_1\cdots e_{n-1}e_n^{(2)}e_{n-1} \cdots e_2e_1e_0(\tilde u_{\varpi_n}^1)=(q^{-1} y) \tilde u_{\lambda}$,
\item[{\rm (d)}] $e_1\cdots e_{n-1}e_n^{(2)}e_{n-1} \cdots e_2e_1e_0(\tilde u_{\varpi_n}^2)=(q^{-1}x) \tilde u_{\lambda}$,
\end{enumerate}
 in $V(\va_1)_y \otimes V(\va_n)_x$.
\end{lemma}

From these lemmas, we obtain
\begin{align*}
 \left( \begin{matrix} q^{-1} y^{-1} & (-1)^n q^{2n}x^{-1} \\ y & q^{-1} x \end{matrix} \right) (a_{ij}^{\varpi_n})
=  \left( \begin{matrix} y^{-1} & (-1)^n q^{-2n-2} x^{-1} \\ q^{-1}y & q^{-1} x \end{matrix} \right),
\end{align*}
and hence
\begin{align*}
(a_{ij}^{\varpi_n}) = \dfrac{1}{z^2 + \mqs^{n+1}}
\left( \begin{matrix} qz^2 - (-1)^n q^{2n+1}  & (-1)^n (q^{-2n-1} - q^{2n+1}) z
 \\ (1-q^2)z& z^{2} - (-1)^{n} q^{-2n} \end{matrix} \right),
\end{align*} where $z=xy^{-1}$.

Hence we can conclude that
\begin{align} \label{eq:deno1n D}
  d_{1,n}(z)=d_{n,1}(z)= z^2+\mqs^{n+1}  \quad \text{ for } \g = D^{(2)}_{n+1}.
\end{align}

\subsection{Denominators between fundamental representations} Write
$$ d_{k,l}(z) =  \prod_{\nu}(z- x_\nu).$$

For rational functions $f,g \in \ko(z)$, we write $f \equiv g$ if there exists an element $a \in \ko[z^{\pm 1}]^\times$ such that
$$ f = ag.$$

\begin{lemma} \label{Lem: aij and dij} \cite{AK}
For $k,l \in I_0$, we have
\begin{equation} \label{eq: aij and dij}
\begin{aligned}
& a_{k,l}(z)a_{k,l}((p^*)^{-1}z)\equiv
\dfrac{d_{k,l}(z)}{d_{k,l}(p^*z^{-1})},\\
 & a_{k,l}(z)
 = q^{(\varpi_k , \varpi_l)}
 \prod_{\nu}\dfrac{(p^* x_\nu z; p^{*2})_\infty
 (p^* x_\nu^{-1} z; p^{*2})_\infty}
  {(x_\nu z; p^{*2})_\infty (p^{*2} x_\nu^{-1} z; p^{*2})_\infty},
\end{aligned}
\end{equation}
where $(z;q)_\infty= \prod_{s=0}^\infty (1-q^sz)$.
\end{lemma}

Now we list a table for triple ($\delta$,$c$,$p^*$) for each $\g$:
\fontsize{10}{10}\selectfont
\begin{table}[ht]
\renewcommand{\arraystretch}{1.3}
\centering
\begin{tabular}[c]{|c|c|c|c|} \hline
$\g$ & $\delta$  & $c$ & $p^*$ \\ \hline
$A^{(2)}_{2n-1}$ &  $\alpha_0+\alpha_1+2(\alpha_2+\cdots +\alpha_{n-1})+\alpha_{n}$  & $h_0+h_1+2(h_2+\cdots+h_{n})$  & $-(-q)^{2n}$ \\ \hline
$A^{(2)}_{2n}$& $2(\alpha_0+\cdots+\alpha_{n-1})+\alpha_{n}$ & $h_0+2(h_1+\cdots+h_{n})$ & $(-q)^{2n+1}$ \\ \hline
$B^{(1)}_{n}$ & $\alpha_0+\alpha_1+2(\alpha_2+\cdots +\alpha_{n})$ & $h_0+h_1+2(h_2+\cdots+h_{n-1})+h_n$ & $-(-q)^{2n-1}$ \\ \hline
$D^{(2)}_{n+1}$ & $\alpha_0+\alpha_1+\cdots +\alpha_{n}$ & $h_0+2(h_1+\cdots+h_{n-1})+h_n$ & $-\mqs^n$  \\ \hline
\end{tabular}

\caption{ ($\delta$,$c$,$p^*$) for each affine type}
\label{table:p^*}
\end{table}
\fontsize{11}{11}\selectfont

By Lemma \ref{Lem: aij and dij} and \eqref{eq: denominator d11}, we can compute
$a_{1,1}(z)$ for all $\g$ as follows:
\begin{align}
 a_{1,1}(z) &=  \begin{cases}  q\dfrac{\langle 2n+2 \rangle \langle 2n-2 \rangle}{\langle 2n \rangle^2}
                    \dfrac{[4n][0]}{[2][4n-2]} & \text{if $\g=A^{(2)}_{2n-1}$, }  \\[2ex]
                   q\dfrac{[2n+3][2n-1]}{[2n+1]^2}\dfrac{[4n+2][0]}{[2][4n]} & \text{if $\g=A^{(2)}_{2n}$,} \\[2ex]
                   q\dfrac{[2n+1][2n-3]}{[2n-1]^2}\dfrac{[4n-2][0]}{[2][4n-4]} & \text{if $\g=B^{(1)}_{n}$,} \\[2ex]
                   q\dfrac{ \{ n+1 \} \{ n-1 \} }{ \{ n \}^2 }\dfrac{ \{ 2n \} \{ 0 \} }{ \{ 1 \} \{ 2n-1 \} } & \text{if $\g=D^{(2)}_{n+1}$,}
                   \end{cases}   \label{eq: a11}
\end{align}
where, for $a \in \Z$ and $b \in \dfrac{1}{2} \Z$,
$$[a]=((-q)^{a}z ; p^{*2})_\infty, \quad \langle a \rangle=(-(-q)^{a}z
; p^{*2})_\infty \quad \text{ and } \quad \{ b \} =(\mqs^{b}z ; p^{*2})_\infty \times (-\mqs^{b}z ;
p^{*2})_\infty .$$

Note that, for $a \in \Z$ and $b \in \frac{1}{2} \Z$, we have
\begin{equation} \label{eq: division value}
\begin{aligned}
&[a]/[a+4n] \equiv z-(-q)^{-a} \text{ and } \langle a \rangle /\langle a+4n \rangle \equiv z+(-q)^{-a} \quad \ \ \ \text{ if $\g=A^{(2)}_{2n-1}$}, \\
& \qquad \qquad \qquad \qquad [a]/[a+4n+2] \equiv z-(-q)^{-a} \qquad\qquad\qquad \ \text{ if $\g=A^{(2)}_{2n}$}, \\
& \qquad \qquad \qquad \qquad [a]/[a+4n-2] \equiv z-(-q)^{-a} \qquad\qquad\qquad \ \text{ if $\g=B^{(1)}_{n}$}, \\
& \qquad \qquad \qquad \qquad \{ b \} / \{ b+2n  \} \equiv z^2-\mqs^{-2b} \qquad\qquad\qquad \ \text{ if $\g=D^{(2)}_{n+1}$}. \end{aligned}
\end{equation}

\medskip

Following \cite{JMO00}, \cite[(3.12)]{DO96} and \cite[(3.7)]{J86}, we recall the image of $v_k \otimes v_l$ ($k \ne l \in I_0$) under the
normalized $R$-matrix
\begin{align*}
  \Rnorm_{1,1}(z) \colon
   V(\varpi_1) \otimes V(\varpi_1)_z \rightarrow V(\varpi_1)_z \otimes V(\varpi_1),
\end{align*}
which is given by
\begin{align} \label{eq:rmatrix 11}
  \Rnorm_{1,1}(z)(v_k \otimes v_l) =
   \dfrac{(1-(q^2)^t) z^{t\times\delta(k \succ l)}}{z^t-(q^2)^t} (v_k \otimes v_l)
   + \dfrac{q^t(z^t-1)}{z^t-(q^2)^t} (v_l \otimes v_k).
\end{align}
Here $\succ$ is the linear order on the labeling set of the basis of $V(\va_1)$ (see \cite[Section 8]{HK02}).


\begin{proposition} \label{prop: akl}
  For $1 \le k,l \le n-\vartheta$, we have
  \begin{align} \label{eq: akl}
    a_{k,l} (z) \equiv
    \begin{cases}
    \dfrac{[|k-l|][4n-|k-l|]}{[k+l][4n-k-l]}
    \dfrac{\langle 2n+k+l \rangle \langle 2n-k-l \rangle }{\langle 2n+|k-l| \rangle \langle 2n-|k-l| \rangle} & \text{if $\g=A^{(2)}_{2n-1}$, } \\[2ex]
    \dfrac{[|k-l|][4n+2-|k-l|]}{[k+l][4n+2-k-l]}\dfrac{[2n+1+k+l][2n+1-k-l]}{[2n+1+|k-l|][2n+1-|k-l|]}  & \text{if $\g=A^{(2)}_{2n}$}, \\[2ex]
    \dfrac{[|k-l|][2n+k+l-1][2n-k-l-1][2n-|k-l|-1]}{[k+l][2n+k-l-1][2n-k+l-1][2n-k-l-2]} & \text{if $\g=B^{(1)}_{n}$},  \\[2ex]
    \dfrac{ \{|\frac{k-l}{2}| \}\{2n-|\frac{k-l}{2}|\}\{n+\frac{k+l}{2}\}\{n-\frac{k+l}{2}\}}
{\{\frac{k+l}{2}\}\{2n-\frac{k+l}{2}\}\{n+|\frac{k-l}{2}|\}\{n-|\frac{k-l}{2}|\}} & \text{if $\g=D^{(2)}_{n+1}$}.  \\[2ex]
    \end{cases}
  \end{align}
\end{proposition}

\begin{proof}
We prove only for the case when $\g$ is of type $A^{(2)}_{2n-1}$. For the other $\g$, one can apply the same argument
to prove our assertion. We first consider when $k=1$.

By \eqref{eq: a11}, our assertion for $k=l=1$ holds. Applying the commutative diagram \eqref{Diagram: Runiv} for $k=1$,
we have
\begin{align} \label{eq: mapsto}
a_{1,l-1}(-q^{-1}z) a_{1,1}((-q)^{l-1}z)v_{[1,\ldots, l-1]} \otimes w \longmapsto  a_{1,l}(z) v_{[1,\ldots, l-1,l]} \otimes v_1,
\end{align}
where $$w=\Rnorm_{1,1}((-q)^{l-1}z)( v_1 \otimes v_{l})
 =  \dfrac{q((-q)^{l-1}z -1 )}{(-q)^{l-1}z -q^2}v_{l} \otimes v_1 +
\dfrac{(1 -q^2 )}{(-q)^{l-1}z -q^2} v_{1} \otimes v_l.$$
Since $v_{[1,\ldots,l-1]} \otimes v_1$ vanishes under the map $p_{l-1,1}$, \eqref{eq: mapsto} indicates that
\begin{align*}
a_{1,l}(z) &= a_{1,l-1}(-q^{-1}z) \ a_{1,1}((-q)^{l-1}z)
 \dfrac{q((-q)^{l-1}z -1 )}{(-q)^{l-1}z -q^2} \\
&\equiv a_{1,l-1}(-q^{-1}z) \ a_{1,1}((-q)^{l-1}z)
\dfrac{[l-1]}{[4n+l-1]}\dfrac{[4n+l-3]}{[l-3]}.
\end{align*}
Hence our assertion for $k=1$ follows from an induction on $l$:
\begin{align} \label{eq:a1k A}
  a_{1,l}(z)= a_{l,1}(z) \equiv\dfrac{[l-1][4n-l+1]}{[l+1][4n-l-1]}
  \dfrac{\langle 2n-l-1 \rangle \langle 2n+l+1 \rangle}{\langle 2n+l-1 \rangle \langle 2n-l+1 \rangle }.
\end{align}

By \eqref{eq: symmetric}, we now assume $2 \le l \le k \le n$. By the direct calculation, one can show that
\begin{align*}
f_{l-1} f_{l-2} \cdots f_1 (v_{[1,\dots,k]} \otimes v_1) =   v_{[1,\dots,k]} \otimes v_l \quad\text{and}\quad
f_{l-1} f_{l-2} \cdots f_1 (v_1 \otimes v_{[1,\dots,k]} ) =  v_l \otimes  v_{[1,\dots,k]}.
\end{align*}
Since $\Rnorm_{k,1}$ is a $\uqpg$-homomorphism and sends $v_{[1,\ldots,k]} \otimes v_1$ to $v_1 \otimes v_{[1,\ldots,k]}$, we have
$$\Rnorm_{k,1}(z)(v_{[1,\ldots,k]} \otimes v_l) =  v_l \otimes  v_{[1,\dots,k]}.$$
Thus, the image in \eqref{Diagram: Runiv2},
\begin{align*}
a_{k,l-1}(-q^{-1}z) a_{k,1}((-q)^{l-1}z)v_{[1,\ldots, l-1]} \otimes w \longmapsto a_{k,l}(z) v_{[1,\ldots, l-1,l]} \otimes v_{[1,\ldots,k]}
\end{align*}
for $w=\Rnorm_{k,1}((-q)^{l-1}z)( v_{[1,\dots,k]} \otimes v_l ) =  v_l \otimes  v_{[1,\dots,k]}$, implies that
\begin{align} \label{eq:recursive akl A}
a_{k,l}(z) = a_{k,l-1}(-q^{-1}z) \ a_{k,1}((-q)^{l-1}z) &\quad (2
\leq l \le k \le n).
\end{align}
Hence one can obtain our assertion by applying an induction on $l$.
\end{proof}

\medskip


\begin{theorem} \label{Thm: dkl}
For $1\leq k,l \le n-\vartheta$, we have
\begin{align}\label{eq: dkl}
 d_{k,l}(z) =  \prod_{s=1}^{\min(k,l)} (z^t-(-q^t)^{|k-l|+2s})(z^t-(p^*)^t(-q^t)^{2s-k-l}).
\end{align}
\end{theorem}

\begin{proof}
For $1 \le k, l \le n-\vartheta$, set
\begin{align} \label{eq: D(k,l)}
  \Dkl(z) = \prod_{s=1}^{\min(k,l)} (z^t-(-q^t)^{|k-l|+2s})(z^t-(p^*)^t\mqs^{2s-k-l}).
\end{align}

Then we can observe that $\Dkl(z)$ behaves similar to $d_{k,l}(z)$.
Namely, (cf. \eqref{eq: symmetric}, \eqref{eq: denominator d11} and \eqref{eq: aij and dij})
\begin{align}
& D_{1,1}(z)=d_{1,1}(z), \qquad  \Dkl(z)=D_{l,k}(z),  \label{eq: sym D and start}  \\
& \dfrac{\Dkl(z)}{\Dkl(p^*z^{-1})} \equiv a_{k,l}(z)a_{k,l}((p^*)^{-1}z) \equiv \dfrac{d_{k,l}(z)}{d_{k,l}(p^*z^{-1})}. \label{eq: Dd}
\end{align}

By calculations, one can check that
\begin{align} \label{eq:d_k,l d_k,l-1 with k1}
  \Dkl(z) = D_{k,l-1}((-q^t)^{-1/t}z)D_{k,1}((-q^t)^{l-1/2}z) \quad \text{ for } 2 \le k \le n-\vartheta,
\end{align}
which is similar to \eqref{eq:recursive akl A}, also.

\medskip

Now we give a proof for $\g=D^{(2)}_{n+1}$, since
this case is most complicated. For the other $\g$, one can apply the similar argument to prove.

\medskip

We shall show that $\Dkl(z)=d_{k,l}(z)$ indeed. Our assertion for $k=l=1$ is presented in \eqref{eq:deno1n D}.
Assume that $1 \le k \le n-1$ and $2 \le l \le n-1$.

From the a surjective homomorphism in Theorem \ref{thm: pij iota ij 1}
\begin{align*}
  p_{l-1,1} \colon V(\varpi_{l-1})_{\mqs^{-\frac{1}{2}}} \otimes V(\varpi_{1})_{\mqs^{\frac{l-1}{2}}}
  \twoheadrightarrow V(\varpi_{l}),
  \end{align*}
the first formula in Lemma \ref{lem:dvw avw} with setting $W=V(\va_k)$ yields an element in $\ko[z^{\pm 1}]$ as follows:
\begin{align}\label{eq:dk1}
\dfrac{d_{k,l-1}(\mqs^{-\frac{1}{2}}z)d_{k,1}(\mqs^{\frac{l-1}{2}}z)}{d_{k,l}(z)} \dfrac{a_{k,l}(z)}
{a_{k,l-1}(\mqs^{-\frac{1}{2}}z)a_{k,1}(\mqs^{\frac{l-1}{2}}z)} \in \ko[z^{\pm 1}],
\end{align}
for $1 \le k \le n-1$ and  $2 \le l \le n-1$. In particular, if $2 \le l \le k \le n-1$,
\begin{align} \label{eq:dkl induction argu}  \dfrac{d_{k,l-1}(\mqs^{-\frac{1}{2}}z)d_{k,1}(\mqs^{\frac{l-1}{2}}z)}{d_{k,l}(z) } \in \ko[z^{\pm 1}],
\end{align}
since (cf. \eqref{eq:recursive akl A})
$$\dfrac{a_{k,l}(z)}{a_{k,l-1}(\mqs^{-\frac{1}{2}}z)a_{k,1}(\mqs^{\frac{l-1}{2}}z)} \in \ko[z^{\pm 1}]^\times$$
by the computation using \eqref{eq: akl}.

Using \eqref{eq: akl} once again, for $k=1 < l$, one can compute that
$$\dfrac{a_{1,l}(z)}{a_{1,l-1}(\mqs^{-\frac{1}{2}}z) a_{1,1}(\mqs^{\frac{l-1}{2}}z)} \equiv
\dfrac{(z^2-\mqs^{1-l})}{(z^2-\mqs^{3-l})} \quad \text{for $2 \le l \le n-1$.}$$

Set $k=1$ and then replace $l$ with $k$ in \eqref{eq:dk1}. Then \eqref{eq:dk1} becomes
\begin{equation}\label{eq:d1l divides blabla D}
\begin{aligned}
 &\dfrac{d_{1,k-1}(\mqs^{-\frac{1}{2}}z)D_{1,1}(\mqs^{\frac{k-1}{2}}z)}{d_{1,k}(z)}
  \dfrac{(z^2-\mqs^{1-k})}{(z^2-\mqs^{3-k})} \\[1ex]
 &\equiv \dfrac{d_{1,k-1}(\mqs^{-\frac{1}{2}}z)(z^2-\mqs^{2n-k+1})
  (z^2-\mqs^{1-k})}{d_{1,k}(z)}
    \in \ko[z^{\pm 1}]
    \quad \text{for $2 \le k\leq n-1$,}
\end{aligned}
\end{equation}
since $D_{1,1}(z)=d_{1,1}(z)$.

On the other hand, from the surjective homomorphism
\begin{align*}
V(\varpi_k)_{\mqs^{-\frac{1}{2}}} \otimes
V(\varpi_1)_{-\mqs^{\frac{2n-k}{2}}} \rightarrow V(\varpi_{k-1}),
  \end{align*}
the second formula in Lemma \ref{lem:dvw avw} with setting $W=V(\va_l)$ yields an element in $\ko[z^{\pm 1}]$ as follows:
\begin{align} \label{eq: D 2}
\dfrac{d_{1,l}(-\mqs^{\frac{k-2n}{2}}z)d_{k,l}(\mqs^{\frac{1}{2}}z)}{d_{k-1,l}(z)}
\dfrac{a_{k-1,l}(z)}{ a_{k,l}(\mqs^{\frac{1}{2}}z)
a_{1,l}(-\mqs^{\frac{k-2n}{2}}z)} \in \ko[z^{\pm 1}].
\end{align}

By computation using \eqref{eq: akl}, we have
\begin{align*}
& \dfrac{a_{k-1,l}(z)}
 {a_{k,l}(\mqs^{\frac{1}{2}}z) a_{1,l}(-\mqs^{\frac{k-2n}{2}}z)}  \equiv \begin{cases}
  \dfrac{z^2-\mqs^{2n-k-l-1}}{z^2-\mqs^{2n-k-l+1}}
  & \text{if}  \ 1 \le l < k \le n-1, \\[2ex]
  \dfrac{(z^2-\mqs^{2n-k-l-1})}{(z^2-\mqs^{2n-k-l+1})}
 \dfrac{( z^2-\mqs^{-1} )}{( z^2-\mqs^{1} )}
  & \text{if} \ 2 \le l = k \le n-1.
 \end{cases}
\end{align*}
Thus the element \eqref{eq: D 2} in $\ko[z^{\pm 1}]$ can be written as follows:
\begin{align}\label{eq:d relarion l-k}
&\dfrac{d_{1,l}(-\mqs^{\frac{k-2n}{2}}z)d_{k,l}(\mqs^{\frac{1}{2}}z)
}{d_{k-1,l}(z)}
  \dfrac{(z^2-\mqs^{2n-k-l-1})}{(z^2-\mqs^{2n-k-l+1})}
 \in \ko[z^{\pm 1}] && \text{if $1 \le l  < k \le n-1$,}
  \end{align}
and
\begin{equation}\label{eq:d relarion 1-1}
\begin{aligned}
&\dfrac{d_{1,l}(-\mqs^{\frac{k-2n}{2}}z)d_{k,l}(\mqs^{\frac{1}{2}}z)
}{d_{k-1,l}(z)} \dfrac{(z^2-\mqs^{2n-k-l-1})}{(z^2-\mqs^{2n-k-l+1})}
 \dfrac{( z^2-\mqs^{-1} )}{( z^2-\mqs^{1} )}
  \in \ko[z^{\pm 1}]
  \end{aligned}
  \end{equation}
if $2 \le l = k \le n-1$.

Setting $l=1$ in \eqref{eq:d relarion l-k}, we obtain
\begin{align} \label{eq:dk1 l=1}
 \dfrac{ D_{1,1}(-\mqs^{\frac{k-2n-1}{2}}z)  d_{k,1}(z) }{d_{k-1,1}(\mqs^{-\frac{1}{2}}z)}
 \dfrac{(z^2-\mqs^{2n-k-1})}{(z^2-\mqs^{2n-k+1})}
\quad \in \ko[z^{\pm 1}] &\quad \text{for $2 \le k \le n-1$.}
\end{align}

Now we claim that
$$d_{1,k}(z)=D_{1,k}(z)=(z^2-\mqs^{k+1})(z^2-\mqs^{2n-k+1}) \quad \text{ for $2 \le k \le n-1$.}$$
With \eqref{eq: sym D and start}, we can start an induction on $k$. Thus \eqref{eq:d1l divides blabla D} can be written in
the following form:
\begin{equation}\label{eq:d1k}
\begin{aligned}
 &\dfrac{D_{1,k-1}(\mqs^{-\frac{1}{2}}z)(z^2-\mqs^{2n+1-k})(z^2-\mqs^{1-k})}{d_{1,k}(z)} \qquad \text{for  $2 \le k \le n-1$}\\
 \equiv &\dfrac{(z^2-\mqs^{k+1})(z^2-\mqs^{2n-k+3})
 (z^2-\mqs^{2n+1-k})(z^2-\mqs^{1-k})}{d_{1,k}(z)}
     \in \ko[z^{\pm 1}].
\end{aligned}
\end{equation}

Now we claim that
\begin{equation} \label{claim D}
\text{$z=\pm\mqs^{\frac{1-k}{2}},\pm\mqs^{\frac{2n-k+3}{2}}$ are not zero of $d_{1,k}(z)$.}
\end{equation}

If \eqref{claim D} is true, we have
\begin{align} \label{eq:dl1 div D}
\dfrac{D_{1,k}(z)}{d_{1,k}(z)}=\dfrac{(z^2-\mqs^{k+1})(z^2-\mqs^{2n+1-k})}{d_{1,k}(z)}
    \in \ko[z^{\pm 1}]\quad (2 \le k \le n-1).
\end{align}

Since $\dfrac{1-k}{2} \le 0$, $\pm\mqs^{\frac{1-k}{2}}\not\in \C[[q]]$. Then \cite[Theorem 2.2.1 (i)]{KKK13b} tells that
$\pm\mqs^{\frac{1-k}{2}}$ can not be a zero of $d_{1,k}(z)$.

If  $z=\pm\mqs^{\frac{k-3}{2}}$ is a zero of $d_{1,k}(z)$, we have a contradiction to the fact that the element \eqref{eq:dl1 div D}
is in $\ko[z^{\pm 1}]$. Thus we know that $z=\pm\mqs^{\frac{k-3}{2}}$ is not a zero of
$d_{1,k}(z)$. Since $\dfrac{\Dkl(z)}{d_{k,l}(z)}  \equiv \dfrac{\Dkl(p^*z^{-1})}{d_{k,l}(p^*z^{-1})}$ by \eqref{eq: Dd},
one can check that
\begin{itemize}
\item $z=\pm\mqs^{\frac{k-3}{2}}$ is not a pole of $D_{1,k}(z)/d_{1,k}(z)$,
\item $z=\pm\mqs^{\frac{2n-k+3}{2}}$ is not a pole of $D_{1,k}(-\mqs^{n}z^{-1})/
d_{1,k}(-\mqs^{n}z^{-1})$.
\end{itemize}
Thus $\pm\mqs^{\frac{2n-k+3}{2}}$ can not be zero of $d_{1,k}(z)$ and hence the claim in \eqref{claim D} holds.

By an induction on $k$ in
\eqref{eq:dk1 l=1},
we also obtain
\begin{align*}
  &\dfrac{d_{1,k}(z)}{(z^2-\mqs^{k+1})(z^2-\mqs^{2n+1-k})}
  \quad \in \ko[z^{\pm 1}]  &&\quad \text{if} \ k \neq n-1, \\
&\dfrac{d_{1,k}(z)}{(z^2-\mqs^{2n+1-k})}
  \quad \in \ko[z^{\pm 1}]  &&\quad \text{if} \ k = n-1.
\end{align*}
By Theorem \ref{thm: pij iota ij 1} and Lemma
\ref{lemma: k,l to n,n}, $d_{1,k}(z)$ has zeros at $\pm\mqs^{\frac{k+1}{2}}$
for $1 \le k \le n-1$.
Thus we have \begin{align} \label{eq:dl1 rev div D}
\dfrac{d_{1,k}(z)}{D_{1,k}(z)}=\dfrac{d_{1,k}(z)}{(z^2-\mqs^{k+1})(z^2-\mqs^{2n+1-k})}
\in \ko[z^{\pm 1}]  \  (2\leq k \le n-1). \end{align}
By considering \eqref{eq:dl1 div D} and \eqref{eq:dl1 rev div D} together, our assertion for $k=1$ holds:
\begin{align*}
d_{1,k}(z)=
(z^2-\mqs^{k+1})(z^2-\mqs^{2n+1-k})=D_{1,k}(z)
&\quad  (2 \le k \le n-1).
\end{align*}

Now we apply an induction on $k+l$.
Applying the induction at \eqref{eq:dkl induction argu} with \eqref{eq:d_k,l d_k,l-1 with k1},
we have \begin{align*}
  \dfrac{d_{k,l-1}(\mqs^{-\frac{1}{2}}z)d_{k,1}(\mqs^{\frac{l-1}{2}}z)}{d_{k,l}(z) }
  =\dfrac{D_{k,l-1}(\mqs^{-\frac{1}{2}}z)D_{k,1}(\mqs^{\frac{l-1}{2}}z) }{d_{k,l}(z) }
  =\dfrac{\Dkl(z)}{d_{k,l}(z)} \
   \in \ko[z^{\pm 1}]
\end{align*}
for $2 \le l \le  k \le n-1$.

\medskip

Let $\phi_{k,l}(z)$ be the elements in $\ko[z^{\pm 1}]$ satisfying $\Dkl(z) = d_{k,l}(z)
\phi_{k,l}(z)$. We claim that
$$ \phi_{k,l}(z)=1 \quad \text{ for } \quad 2 \le l
\leq k \le n-1.$$

Note that
\begin{align*}
  &\dfrac{D_{1,l}(-\mqs^{\frac{k-2n}{2}}z)\Dkl(\mqs^{\frac{1}{2}}z) }{D_{k-1,l}(z)}
  \dfrac{ (z^2-\mqs^{2n-k-l-1} ) }{( z^2-\mqs^{2n-k-l+1}) }  \\
  &=\begin{cases}
(z^2-\mqs^{4n-k-l+1})(z^2-\mqs^{2n-k-l-1}) & \text{if $l < k$,}\\[1.5ex]
 (z^2-\mqs^{4n-k-l+1})(z^2-\mqs^{2n-k-l-1})(z^2-\mqs^{2n+1})(z^2-\mqs)
 & \text{if $l=k$.}
   \end{cases}
\end{align*}
By \eqref{eq:d relarion l-k}, \eqref{eq:d relarion 1-1} and an induction on $k+l$, the above elements are written in the following form:
\begin{align*}
&\dfrac{(z^2-\mqs^{4n-k-l+1})(z^2-\mqs^{2n-k-l-1})}{\phi_{k,l}(\mqs^{\frac{1}{2}}z)} \ \in  \ko[z^{\pm 1}] \qquad \text{if} \ l < k,\\
&\dfrac{(z^2-\mqs^{4n-k-l+1})(z^2-\mqs^{2n-k-l-1})
(z^2-\mqs^{2n+1})(z^2-\mqs^{-1})}{\phi_{k,l}(\mqs^{\frac{1}{2}}z)} \in  \ko[z^{\pm 1}]
 \quad \text{if} \ l=k.
\end{align*}

Since $\phi_{k,l}(\mqs^{\frac{1}{2}}z)$ divides
$\Dkl(\mqs^{\frac{1}{2}}z)$,
we conclude that
\begin{align}
& \phi_{k,l}(z) = 1 & \text{ if }  k+l <n, \nonumber \\
& \dfrac{(z^2-\mqs^{2n-k-l})}{\phi_{k,l}(z)} \ \in
\ko[z^{\pm 1}]
& \text{ if }  k+l \ge n. \label{eq:psi D}
\end{align}

Now our assertion holds if $z=\pm\mqs^{\frac{2n-k-l}{2}}$ is
not a zero of $\phi_{k,l}(z)$ for $k+l\ge n$.
From \eqref{eq: Dd}, one can see that
$\phi_{k,l}(-\mqs^{n}z^{-1})\equiv\phi_{k,l}(z)$. Thus we suffice to prove
that $z=\pm\mqs^{\frac{k+l}{2}}$ is not a zero of
$\phi_{k,l}(z)$ for $k+l\ge n$. If $k+l>n$, then we have $n > 2n-k-l$ and hence $\phi_{k,l}(z)=1$.

Now we consider when $k+l=n$. Then Lemma \ref{lemma: k,l to n,n} tells that $d_{k,l}(z)$ has zeros at
$z=\pm\mqs^{\frac{k+l}{2}}$. By the definition of $\Dkl(z)$,
$\pm\mqs^{\frac{k+l}{2}}$ is a zero of multiplicity $1$. Thus $\pm\mqs^{\frac{k+l}{2}}$ can not be a zero of
$\phi_{k,l}(z)$ when $k+l=n$.
\end{proof}

Now we shall compute $d_{k,n}(z)$ for $\g=B^{(1)}_n$ and $\g=D^{(2)}_{n+1}$.
By Lemma \ref{Lem: aij and dij}, \eqref{eq:deno1n B}
and \eqref{eq:deno1n D}, we have
\begin{align}
 a_{1,n}(z) & \equiv  \begin{cases} \dfrac{[2n-3]_{(n+1)}[6n-1]_{(n+1)}}{[2n+1]_{(n+1)}[6n-5]_{(n+1)}} & \text{if $\g=B^{(1)}_{n}$,} \\[2ex]
                   \dfrac{ \{ \frac{3n+1}{2} \}' \{ \frac{n-1}{2} \}' }{ \{ \frac{3n-1}{2} \}' \{\frac{n+1}{2} \}' } & \text{if $\g=D^{(2)}_{n+1}$,}
                   \end{cases}   \label{eq: a1n}
\end{align}
where, for $a,k \in \Z$ and $b \in \dfrac{1}{2}\Z$,
$$[a]_{(k)}=((-1)^kq_s^{a}z ; p^{*2})_\infty  \quad \text{ and } \quad \{ b\}' = ( -\sqrt{-1} \mqs^b; p^{*2})_\infty (\sqrt{-1}\mqs^b; p^{*2})_\infty.$$

Now, we give a proof only for $\g=B^{(n)}_1$. For $\g=D^{(2)}_{n+1}$, one can apply the same arguments.

\begin{proposition} \label{proposition: a_k,n }
  For $1\leq l \le  n-1$, we have
  \begin{align} \label{eq:aln}
    a_{l,n}(z) \equiv \begin{cases} \dfrac{[2n-2l-1]_{(n+l)}[6n+2l-3]_{(n+l)}}{[2n+2l-1]_{(n+l)}[6n-2l-3]_{(n+l)}} & \text{ if } \g =B^{(1)}_n, \\
\dfrac{ \{ \frac{3n+l}{2} \}'  \{\frac{n-l}{2}\}' }{  \{ \frac{3n-l}{2} \}'  \{\frac{n+l}{2}\}' } & \text{ if } \g =D^{(2)}_{n+1}. \end{cases}
  \end{align}
\end{proposition}

\begin{proof}
By \eqref{eq: a1n}, it suffices to consider when $2 \le l \le n-1$.
Applying the commutative diagram \eqref{Diagram: Runiv} with setting $k=n$, \eqref{Diagram: Runiv2} tells that we have
\begin{align*}
a_{n,l-1}(-q^{-1}z) a_{n,1}((-q)^{l-1}z)v_{[1,\ldots, l-1]} \otimes w \longmapsto a_{n,l}(z) v_{[1,\ldots, l-1,l]} \otimes m^+_{n},
\end{align*}
where $w=\Rnorm_{n,1}((-q)^{l-1}z)(m^+_{n} \otimes v_{l})$ for the highest weight vector $m_n^+$ of $V(\va_n)$.

Since $m^+_{n}$ vanishes by the action $f_i$ $( 1 \le i \le l-1)$, as in the proof of Proposition \ref{prop: akl},
\begin{align*}
w= \Rnorm_{n,1}((-q)^{l-1}z)(m^+_{n} \otimes v_{l})= v_{l} \otimes m^+_{n},
\end{align*}
and hence
\begin{align} \label{eq:recursive formula a_nk}
a_{n,l}(z) = a_{n,l-1}(-q^{-1}z) \ a_{n,1}((-q)^{l-1}z) \quad \text{ for $2 \le l \le n-1$}.
\end{align}
By \eqref{eq: a1n} and an induction on $l$, our assertion follows.
\end{proof}

\begin{theorem}
  For $1 \le k \le n-1$, we have
  \begin{align}\label{eq:dkn}
    d_{k,n}(z)= \begin{cases} \displaystyle \prod_{s=1}^{k}(z-(-1)^{n+k}q_s^{2n-2k-1+4s}) & \text{ if } \g =B^{(1)}_n, \\
                              \displaystyle \prod_{s=1}^{k}(z^2+(-q^{2})^{n-k+2s}) & \text{ if } \g =D^{(2)}_{n+1}. \end{cases}
  \end{align}
\end{theorem}

\begin{proof}
By \eqref{eq:deno1n B}, it suffices to consider when $2 \le k \le n-1$.
From the surjective homomorphism in Theorem \ref{thm: pij iota ij 1}
$$V(\varpi_{k-1})_{(-q)^{-1}} \otimes  V(\va_1)_{(-q)^{k-1}}\twoheadrightarrow V(\va_k),$$
the first formula in Lemma \ref{lem:dvw avw} with $W=V(\va_n)$ yields an element in $\ko[z^{\pm 1}]$ as follows:
\begin{align*}
    \dfrac{d_{k-1,n}(-q^{-1}z) d_{1,n}((-q)^{k-1}z)}{d_{k,n}(z)} \dfrac{a_{k,n}(z)}
  { a_{k-1,n}(-q^{-1}z) a_{1,n}((-q)^{k-1}z)} \in \ko[z^{\pm 1}].
  \end{align*}
By \eqref{eq:recursive formula a_nk}, the element is written in more simplified form as follows:
\begin{align} \label{eq:dkn form 1}
    \dfrac{d_{k-1,n}(-q^{-1}z) d_{n,1}((-q)^{k-1}z) }
  {d_{k,n}(z)}
 \equiv\dfrac{d_{k-1,n}(-q^{-1}z) (z-(-1)^{n+k}q_s^{2n-2k+3}) }
  {d_{k,n}(z)}
    \in \ko[z^{\pm 1}].
  \end{align}
On the other hand, for each $2 \le k \le n-1 $, we have a surjective
homomorphism
\begin{align*}
V(\varpi_k)_{-q^{-1}} \otimes V(\varpi_1)_{-(-q)^{2n-1-k}} \twoheadrightarrow
V(\varpi_{k-1}).
  \end{align*}
Then the second formula in Lemma \ref{lem:dvw avw} with $W=V(\va_n)$ yields an element in $\ko[z^{\pm 1}]$ as follows:
\begin{align} \label{eq: second b}
    \dfrac{d_{1,n}(-(-q)^{k+1-2n}z) d_{k,n}(-qz)}{d_{k-1,n}(z)}\dfrac{ a_{k-1,n}(z)}
  { a_{1,n}(-(-q)^{k+1-2n}z) a_{k,n}(-qz)} \in \ko[z^{\pm 1}].
  \end{align}
Using \eqref{eq:aln}, the second factor of \eqref{eq: second b} can be written as
\begin{align*}
    \dfrac{ a_{k-1,n}(z)}
  {a_{1,n}(-(-q)^{k+1-2n}z) a_{k,n}(-qz) }
  \equiv
  \dfrac{z-(-1)^{n+k+1}q_s^{2n-2k-3}}{z-(-1)^{n+k+1}q_s^{2n-2k+1}}.
\end{align*}
and hence \eqref{eq: second b} becomes
    \begin{align*}
    \dfrac{ d_{k,n}(-qz) }
  {d_{k-1,n}(z)}
  \dfrac{(z-(-1)^{n+k+1}q_s^{6n-2k-1})(z-(-1)^{n+k+1}q_s^{2n-2k-3})}{(z-(-1)^{n+k+1}q_s^{2n-2k+1})} \in \ko[z^{\pm 1}].
  \end{align*}
By the induction hypothesis, $z=(-1)^{n+k+1}q_s^{6n-2k-1}$ and $(-1)^{n+k+1}q_s^{2n-2k-3}$ are not zeros of $d_{k-1,n}(z)$.
Hence we can conclude that
  \begin{align} \label{eq:dkn form 2}
    \dfrac{ d_{k,n}(-qz) }
  {d_{k-1,n}(z)(z-(-1)^{n+k+1}q_s^{2n-2k+1})} \in \ko[z^{\pm 1}],
  \end{align}
which is equivalent to
\begin{align*}
    \dfrac{ d_{k,n}(z) }
  {d_{k-1,n}(-q^{-1}z)(z-(-1)^{n+k} q_s^{2n-2k+3})}\in \ko[z^{\pm 1}].
  \end{align*}
Considering \eqref{eq:dkn form 1} and \eqref{eq:dkn form 2} together, our assertion follows:
\begin{align*}
  d_{k,n}(z)\equiv d_{k-1,n}(-q^{-1}z)(z-(-1)^{n+k}q_s^{2n-2k+3})
  =\prod_{s=1}^{k}(z-(-1)^{n+k}q_s^{2n-2k-1+4s}).
\end{align*}
\end{proof}

\begin{remark}
In conclusion, we can observe that
$$\text{ for all $1 \le k \le k \le n$, $\Rnorm_{k,l}(z)$ has only simple poles unless $\g = D^{(2)}_{n+1}$.}$$
For $\g=D^{(2)}_{n+1}$, $\Rnorm_{k,l}(z)$ has a double pole at $z=\pm(-q^2)^{s/2}$ if
$$ 2 \le k,l \le n-1, \ k+l > n, \ 2n+2-k-l \le s \le k+l \text{ and } s \equiv k+l \mod \ 2.$$
\end{remark}

\newpage
\appendix
\section{The table of denominators.}
\vskip 1em
\begin{center}
\begin{tabular}{ | c | c | c | c  | } \hline
Type & $n$ & $k,l$ &  Denominators \\ \hline
$A_{n}^{(1)}$& $n \ge 1$ & $1 \le k,l \le n$& $ d_{k,l}(z)= \displaystyle\prod_{s=1}^{ \min(k,l,n+1-k,n+1-l)} \big(z-(-q)^{2s+|k-l|}\big)$  \\ \hline
$B_{n}^{(1)}$& $n \ge 3$ & $1 \le k,l \le n-1$ & $d_{k,l}(z) =  \displaystyle \prod_{s=1}^{\min (k,l)} \big(z-(-q)^{|k-l|+2s}\big)\big(z+(-q)^{2n-k-l-1+2s}\big)$ \\ \cline {3-4}
$q_s^2=q$   &  & $1 \le k \le n-1$ & $d_{k,n}(z) = \displaystyle  \prod_{s=1}^{k}\big(z-(-1)^{n+k}q_s^{2n-2k-1+4s}\big)$ \\ \cline {3-4}
            &  & $k=l=n$ & $ d_{n,n}(z)=\displaystyle \prod_{s=1}^{n} \big(z-(q_s)^{4s-2}\big)$ \\ \hline
$C_{n}^{(1)}$& $n \ge 2$ & $1 \le k,l \le n$& $ d_{k,l}(z)= \hspace{-2ex}\displaystyle \hspace{-3ex}\prod_{s=1}^{ \min(k,l,n-k,n-l)} \hspace{-5ex}
\big(z-(-q_s)^{|k-l|+2s}\big)\hspace{-2ex}\prod_{i=1}^{ \min(k,l)}\hspace{-1.5ex} \big(z-(-q_s)^{2n+2-k-l+2s}\big)$   \\ \hline
$D_{n}^{(1)}$& $n \ge 4$ & $1 \le k,l \le n-2$ & $d_{k,l}(z) = \displaystyle \prod_{s=1}^{\min (k,l)}
                \big(z-(-q)^{|k-l|+2s}\big)\big(z-(-q)^{2n-2-k-l+2s}\big)$ \\ \cline {3-4}
             & & $1 \le k \le n-2$ & $ d_{k,n-1}(z)=d_{k,n}(z) = \displaystyle \prod_{s=1}^{k}\big(z-(-q)^{n-k-1+2s}\big)$ \\ \cline {3-4}
             & & $\{k,l\}=\{n,n-1\}$ & $d_{n,n-1}(z)=d_{n-1,n}(z)=\displaystyle \prod_{s=1}^{\lfloor \frac{n-1}{2} \rfloor} \big(z-(-q)^{4s}\big)$\\ \cline {3-4}
             & & $k=l\in\{n,n-1\}$ & $d_{n,n}(z)=d_{n-1,n-1}(z)=\displaystyle \prod_{s=1}^{\lfloor \frac{n}{2} \rfloor} \big(z-(-q)^{4s-2}\big)$ \\ \hline
$A_{2n-1}^{(2)}$& $n \ge 3$ & $1 \le k,l \le n$ & $ d_{k,l}(z)= \displaystyle \prod_{s=1}^{\min(k,l)} \big(z-(-q)^{|k-l|+2s}\big)\big(z+(-q)^{2n-k-l+2s}\big)$ \\ \hline
$A_{2n}^{(2)}$& $n \ge 1$ & $1 \le k,l \le n$ & $d_{k,l}(z) = \displaystyle \prod_{s=1}^{\min(k,l)} \big(z-(-q)^{|k-l|+2s}\big)\big(z-(-q)^{2n+1-k-l+2s}\big)$ \\ \hline
$D_{n+1}^{(2)}$& $n \ge 2$ & $1 \le k,l \le n-1$ &  $d_{k,l}(z) = \displaystyle \prod_{s=1}^{\min(k,l)} \big(z^2 - \mqs^{|k-l|+2s}\big)\big(z^2 - \mqs^{2n-k-l+2s}\big)$ \\ \cline {3-4}
               & & $1 \le k \le n-1$ & $d_{k,n}(z) = \displaystyle \prod_{s=1}^{k}\big(z^2+(-q^{2})^{n-k+2s}\big)$ \\ \cline {3-4}
               & & $k=l=n$ & $ d_{n,n}(z)=\displaystyle \prod_{s=1}^{n} \big(z+\mqs^{s}\big)$ \\ \hline
\end{tabular}
\end{center}

\end{document}